\documentclass[9pt]{IEEEtran}
\usepackage{cite}
\usepackage{amsmath,amssymb,amsfonts}
\usepackage{hyperref}
\usepackage{multirow}
\usepackage{placeins}
 \usepackage{booktabs}
\usepackage[flushleft]{threeparttable}
\usepackage{amsthm}
\usepackage{newtxmath}
\usepackage{mathrsfs}
\usepackage{amsmath}
\usepackage{amsthm}
\usepackage{amssymb}
\usepackage{enumerate}
\usepackage{graphicx} 
\usepackage{float} 
\usepackage{bbm}
\newtheorem{asum}{Assumption}
\newtheorem{lem}{Lemma}
\newtheorem{thm}{Theorem}
\newtheorem{cor}{Corollary}
\newtheorem{rem}{Remark}

\newtheorem{defn}{Definition}

\newtheorem{pro}{Proposition}
\def\argmin{\mathop{\rm arg\,min}}%

\def\BibTeX{{\rm B\kern-.05em{\sc i\kern-.025em b}\kern-.08em
    T\kern-.1667em\lower.7ex\hbox{E}\kern-.125emX}}
\begin{document}
\title{Data-Driven  Distributionally Robust Mixed-Integer Control through 
     Lifted Control Policy}
\author{Xutao Ma, Chao Ning, \IEEEmembership{Senior Member, IEEE}, Wenli Du, \IEEEmembership{Senior Member, IEEE}, and Yang Shi, \IEEEmembership{Fellow, IEEE}
\thanks{This work was supported in part by the National Natural Science Foundation
of China under Grants 62473256 and 62103264, and in part by the National Natural Science Foundation of China (Basic Science Center Program) under Grant 61988101. (\emph{Corresponding author: Chao Ning})}
\thanks{X. Ma and C. Ning are with Department of Automation, Shanghai
Jiao Tong University, Shanghai 200240, China, and Key Laboratory of System
Control and Information Processing, Ministry of Education of China, Shanghai
200240, China (email: maxutao2022@sjtu.edu.cn; chao.ning@sjtu.edu.cn). }
\thanks{W. Du is with the Key Laboratory of Smart Manufacturing in Energy Chemical Process, Ministry of Education, East China University of Science and Technology, Shanghai 200237, China (e-mail: wldu@ecust.edu.cn).}
\thanks{Y. Shi is with the Department of Mechanical Engineering, University of Victoria, Victoria, BC V8W 3P6, Canada (email: yshi@uvic.ca).}
}

\maketitle

\begin{abstract}
This paper investigates the  {finite-horizon distributionally robust mixed-integer control (DRMIC) of uncertain linear systems.} 
{However, deriving an optimal causal feedback control policy to this DRMIC problem is computationally formidable for most ambiguity sets.}
To address the computational challenge, we propose a novel distributionally robust lifted control policy (DR-LCP) method to derive a high-quality approximate solution to this DRMIC problem for a rich class of Wasserstein metric-based ambiguity sets, including the Wasserstein ambiguity set and its variants. { In theory, we analyze the asymptotic performance and establish a tight non-asymptotic bound of the proposed method. In numerical experiments, the proposed DR-LCP method empirically demonstrates superior performance compared with existing methods in the literature.}
\end{abstract}

\begin{IEEEkeywords}
Distributionally robust control, mixed-integer control, Wasserstein metric, lifted control policy.
\end{IEEEkeywords}

\section{Introduction}\label{intro}
\allowdisplaybreaks
In the context of stochastic control, the probability distribution of uncertain disturbance is assumed to be perfectly known. However, in real-world scenarios, the underlying disturbance distribution is only partially observed from the historical uncertainty data. Therefore, any single estimated distribution based on this partial information inevitably deviates from the true one, resulting in poor control performance or even an unstable system \cite{li2023spatiotemporal,nilim2005robust}. 

{To overcome such drawback of stochastic control, distributionally robust control (DRC) has emerged as a paradigm to hedge against this distribution deviation \cite{van2021data,schuurmans2023general,boskos2020data,yang2020wasserstein,ning2021online,van2015distributionally,mark2020stochastic,al2023distributionally,romao2023distributionally,boskos2023data}}. In the literature,
statistical information such as the moment of different orders can be leveraged to construct a DRC problem \cite{ning2021online,van2015distributionally,yang2018dynamic}. More recently, the Wasserstein metric has been favored to cast a Wasserstein DRC problem thanks to its asymptotic convergence and tractable reformulation \cite{yang2020wasserstein,mohajerin2018data,kuhn2019wasserstein,kim2023distributional}, {and many efficient solution algorithms have been developed \cite{cherukuri2019cooperative,cheramin2022computationally}}. With the merit of hedging against distributional ambiguity, DRC has been successfully applied to various control problems, {such as model predictive control \cite{ning2021online,mcallister2023inherent,li2021distributionally,coulson2021distributionally,wu2022ambiguity}, reinforcement learning \cite{zhao2023minimax,panaganti2023distributionally}, and motion control \cite{hakobyan2021wasserstein,zolanvari2023wasserstein}.}
However, the existing research on DRC considers only continuous controls in the system, neglecting the incorporation of integer controls, which model critical control actions in practice, such as the on/off decisions of machines in the process industry \cite{rawlings2017model}.

It is widely acknowledged that deriving an optimal control policy is computationally prohibitive for most DRC problems, due to the need to optimize over the infinite-dimensional space of all the admissible control policies. To address this computational challenge, approximate solution methods have been developed in previous research. Using a parameterized class of feedback control policies is a frequently adopted method to address dynamic control problems \cite{calafiore2009affine,ning2020transformation}. In reference \cite{van2015distributionally}, an affine control policy (ACP) was employed to address moment-based DRC problems. However, the ACP excludes nonlinear control policies and typically leads to over-conservative control actions. For Wasserstein DRC problems, references \cite{yang2020wasserstein} and \cite{kim2023distributional} adopted a Wasserstein penalty relaxation and derived control policies for linear-quadratic Wasserstein DRC problems. { However, their approach can only be applied to quadratic value
function cases, while this condition typically does not hold for mixed-integer control problems.} Therefore, there arises a pressing need to develop a unified approximate solution method that derives high-quality policies for both continuous and integer controls in DRC problems.

{To fill the research gaps, we propose a DR-LCP method by using the lifted control policy (LCP) to solve finite-horizon mixed-integer DRC problems under a class of Wasserstein-type ambiguity sets. 
The LCP was introduced by \cite{georghiou2015generalized} and \cite{bertsimas2018binary}, and the recent work \cite{feng2021multistagea} used the LCP to solve distributionally robust optimization (DRO) problems under relaxed Wasserstein ambiguity sets.  
Going beyond the previous research, we conduct an asymptotic performance analysis and establish a tight non-asymptotic performance bound of the proposed method, which can help us understand the performance of the proposed LCP method theoretically.
Moreover, we develop an equivalent reformulation technique to derive the best LCP of the DRC problems under several Wasserstein-type ambiguity sets (including the standard Wasserstein ambiguity set) without any relaxation.
Numerical experiments on a classic inventory control problem fully verify the superior performance of the proposed DR-LCP method compared with existing methods in the literature. }

{The major contributions of this paper are summarized as
follows.
\begin{itemize}
    \item We propose a novel DR-LCP method to use LCP to solve DRC problems under a class of Wasserstein-type ambiguity sets.
    \item Theoretically, we conduct an asymptotic performance analysis and establish a tight non-asymptotic performance bound of the proposed DR-LCP method.
\end{itemize}}

The rest of the paper is structured as follows. Section \ref{prob} presents the DRMIC problem considered in this paper. The LCP formulation is introduced in Section \ref{lcp}. { We then conduct asymptotic and non-asymptotic analysis of the proposed method in Section {\ref{thero_anal}}.} In Section \ref{refor}, we develop a solution methodology to find the optimal LCP solution to the DRMIC problem for a class of Wasserstein metric-based ambiguity sets. Subsequently, numerical experiments on an inventory control problem are conducted in Section \ref{case}. Conclusions are drawn in Section \ref{concl}.

\noindent\textbf{Notation:}  Throughout the paper, {bold symbols are used exclusively to denote vectors and matrices.} $\mathcal{T}$ is the control horizon. 
$\boldsymbol{\xi}=(\boldsymbol{\xi}_1^T,\cdots,\boldsymbol{\xi}_{\mathcal{T}}^T)^T$ is the overall uncertain disturbance, where
$\boldsymbol{\xi}_t=({\xi}_{t,1},\cdots,{\xi}_{t,n_\xi})^T$ corresponds to $n_\xi$ dimensional disturbance 
realized in time $t$ and ${\xi}_{t,i}$ is the $i$th dimension of $\boldsymbol{\xi}_t$. { $\boldsymbol{\xi}_{[t]}=(\boldsymbol{\xi}_1,\cdots,\boldsymbol{\xi}_t)^T$ represents 
disturbance up to time $t$.}   
$\varXi$, $\varXi_{[t]}$, $\varXi_t$, and $\varXi_{t,i}$ are support of $\boldsymbol{\xi}$, $\boldsymbol{\xi}_{[t]}$, 
$\boldsymbol{\xi}_t$, and ${\xi}_{t,i}$, respectively. The disturbance and support with superscript 
$\ast$ stand for the corresponding lifted versions. $\| \cdot \|$ represents 1-norm and $\| \cdot \|_{\ast}$ is 
its dual norm, \emph{i.e.}, $\infty$-norm. conv$(\cdot)$ is the operator of taking the convex hull. $[N]$ represents the set $\{ 1,\cdots,N \}$.

\section{Problem Statement}
\label{prob}
\allowdisplaybreaks

We consider an uncertain linear system with both continuous and integer controls, given as follows.
\begin{equation}
\label{sys}
{ \boldsymbol{x}_{t+1}=\boldsymbol{A}_{t}\boldsymbol{x}_{t}+\boldsymbol{B}_{t}\boldsymbol{u}_{t}+\boldsymbol{C}_{t}\boldsymbol{\gamma}_{t}+
\boldsymbol{D}_{t}\boldsymbol{\xi}_{t}},
\end{equation}
where $\boldsymbol{x}_t\in \mathbb{R}^{n_x}$ is the state of the system, { $\boldsymbol{u}_t\in \mathbb{R}^{n_u}$ is the continuous control input, $\boldsymbol{\gamma}_t\in {\mathbb{Z}}^{n_\gamma}$ is the integer control input}, and $\boldsymbol{\xi}_{t}\in \varXi_t\subset \mathbb{R}^{n_\xi}$ is the uncertainty in the form of disturbance. System (\ref{sys}) is subject to the linear constraint
\begin{equation}
\begin{gathered}
     \sum_{t=1}^{\mathcal{T}} 
\tilde{\boldsymbol{A}}_{t}\boldsymbol{x}_t+\tilde{\boldsymbol{B}}_{t}
   \boldsymbol{u}_t+
\tilde{\boldsymbol{C}}_{t}\boldsymbol{\gamma}_t+ 
    \tilde{\boldsymbol{D}}_{t} \boldsymbol{\xi}_{t}\leq \boldsymbol{q}, \forall \boldsymbol{\xi} \in \varXi,
\end{gathered}
\label{hard}
\end{equation}
where $\tilde{\boldsymbol{A}}_{t},\tilde{\boldsymbol{B}}_{t},\tilde{\boldsymbol{C}}_{t},\tilde{\boldsymbol{D}}_{t}$ are the coefficient matrices. 
{Concretely, constraint (\ref{hard}) can represent safety requirements, such as the need for the state to remain within a polyhedral safety set under any uncertainty realizations.}

{ In this paper, we assume that the uncertainty support $\varXi$ is a hyperrectangle, which is crucial for the reformulation of the proposed method.}

{ 
\begin{asum}\label{supp_assum}
    The uncertainty support $\varXi$ is a hyperrectangle.
\end{asum}
}
To handle the disturbance of the system, modeling of the joint distribution of $\boldsymbol{\xi}=(\boldsymbol{\xi}_1^T,\cdots,\boldsymbol{\xi}_{\mathcal{T}}^T)^T$ is necessary. However, the true joint distribution of the disturbance is not perfectly known and is only partially observed from the
data $\widehat{\boldsymbol{\xi}}^s,s=1,\cdots,N$. Therefore, we can leverage the information in the data to construct a data-driven ambiguity set $\mathscr{B}(\widehat{\boldsymbol{\xi}}^1,\cdots,\widehat{\boldsymbol{\xi}}^N)$.

Based on the ambiguity set $\mathscr{B}$, we aim to find a causal  disturbance feedback policy $\boldsymbol{\pi}(\cdot)=(\pi^u,\pi^{\gamma})=(\boldsymbol{u}_1^\pi(\cdot),\cdots,\boldsymbol{u}_{\mathcal{T}}^\pi(\cdot),\boldsymbol{\gamma}_{1}^\pi(\cdot),\cdots,\boldsymbol{\gamma}_{\mathcal{T}}^\pi(\cdot))$, with $\boldsymbol{u}_t^\pi(\cdot):\varXi_{[t]}\to \mathbb{R}^{n_u}$ and $\boldsymbol{\gamma}_t^\pi(\cdot):\varXi_{[t]}\to {\mathbb{Z}}^{n_{\gamma}}$ satisfing (\ref{hard}) and minimizing the following {finite-horizon} discounted cost function $J(\boldsymbol{x},\pi)$, which is given by
\begin{equation}
\label{cost}
    J(\boldsymbol{x},\boldsymbol{\pi})=\max_{\mathbb{P}\in \mathscr{B}}\
    \mathbb{E}_{\boldsymbol{\xi}\sim\mathbb{P}}\Bigg[\sum_{t=1}^\mathcal{T} 
    \alpha^t c_t(\boldsymbol{x}_t,\boldsymbol{u}_{t}^\pi,\boldsymbol{\gamma}_t^\pi)\ \Bigg|\ \boldsymbol{x}_0=\boldsymbol{x}
    \Bigg]   
\end{equation}
where $\alpha^t\in [0,1]$ is a discount factor.

Therefore, combining (\ref{sys})-(\ref{cost}), we can formulate the DRMIC problem as 
\begin{equation}
\tag{DRMIC}
\label{drsc}
\begin{gathered}
\min_{\boldsymbol{\pi}\in \Pi}\ \ J(\boldsymbol{x},\boldsymbol{\pi})\ \
    \text{s.t. } (\ref{hard})-(\ref{cost})
\end{gathered}
\end{equation}
where $\Pi$ represents a set of causal disturbance feedback policies.

{ We further define the optimal value of (\mbox{\ref{drsc}}) as $J^\ast (\Pi)$, and for $\Pi=\Pi^u\otimes \Pi^\gamma$, we write the optimal value as $J^\ast (\Pi^u, \Pi^\gamma)$, where $\Pi^u$ and $ \Pi^\gamma$ are spaces of continuous and integer control policies.}

The worst-case expectation term $\max_{\mathbb{P}\in \mathscr{B}}
\mathbb{E}_{\boldsymbol{\xi}\sim\mathbb{P}}\{ \cdot \}$ and the mixed-integer control pose formidable computational challenges in solving the DRMIC problem (\ref{drsc}), and the exact solution method has not been established to date for most ambiguity sets. Therefore, we propose a novel DR-LCP method to derive an approximate solution by restricting causal control policies to the family of LCPs.

\section{Lifted Control Policy}
\label{lcp}
In this section, we begin by introducing the lifting function, which maps the disturbance into the lifted disturbance. Subsequently, we develop LCP for the continuous and integer controllers as affine functions of the continuous and integer parts of the lifted past disturbance, respectively.


{ Before introducing the lifting function, we define some notations regarding the uncertainty support. By Assumption {\ref{supp_assum}}, the uncertainty support $\varXi$ can be expressed as $\varXi=\{ \boldsymbol{\xi} | \boldsymbol{l}\leq \boldsymbol{\xi} \leq \boldsymbol{v}  \}$, where $\boldsymbol{l}$ and $\boldsymbol{v}$ are the lower and upper bounds, respectively. Similarly, 
$\varXi_{[t]}$, $\varXi_t$, and $\varXi_{t,i}$ are hyperrectangles with lower bound $\boldsymbol{l}_{[t]}$, 
$\boldsymbol{l}_t$, ${l}_{t,i}$ and upper bound $\boldsymbol{v}_{[t]}$, 
$\boldsymbol{v}_t$, ${v}_{t,i}$, respectively.}

\subsection{Lifting of Disturbance}
We first introduce the lifting function for each dimension ${\xi}_{t,i}, t\in [\mathcal{T}], i\in [n_\xi]$ of the disturbance.
\begin{defn}
\label{def1}
    Given $p_{t,i}-1$ breakpoints $w_{t,i,j},j\in [p_{t,i}-1]$ satisfying
    \begin{equation}
       {  {l}_{t,i} =w_{t,i,0}\leq w_{t,i,1}\leq \cdots \leq w_{t,i,(p_{t,i}-1)}\leq w_{t,i,p_{t,i}}= {v}_{t,i},} \nonumber
    \end{equation}
    a lifting function $G_{t,i}(\cdot)$ for dimension ${\xi}_{t,i}$ is a one-to-one mapping from the support $\varXi_{t,i}$ into the lifted support $\varXi_{t,i}^\ast$ as
    \begin{equation}
                \begin{aligned}
        G_{t,i}:\varXi_{t,i}&\rightarrow\varXi_{t,i}^{\ast};\ 
    {\xi}_{t,i}\rightarrow G_{t,i}({\xi}_{t,i})=\boldsymbol{\xi}_{t,i}^{\ast}
    \\:=&(V_{t,i,1},\cdots,V_{t,i,p_{t,i}},Q_{t,i,1},\cdots,Q_{t,i,p_{t,i}})^T\\
    =&(\boldsymbol{V}_{t,i}^T,\boldsymbol{Q}_{t,i}^T)^T,
\end{aligned}
    \end{equation}
where 
\setlength{\arraycolsep}{0.5pt}
\begin{equation}
    {  V_{t,i,j}=
    \begin{cases}
    \text{min}\{{\xi}_{t,i},w_{t,i,1}\},\hspace{-4pt} & j=1\\
    \text{max}\{  \text{min}\{ {\xi}_{t,i},w_{t,i,j} \}-w_{t,i,j-1},0 \},\hspace{-4pt}&
        j=2,\cdots,p_{t,i},
    \end{cases}   }
\end{equation}
and for $j\in[p_{t,i}-1]$,
\begin{equation}
    Q_{t,i,j}=\left\{  \begin{array}{ll}
    0, \hspace{10pt} &  {\xi}_{t,i}< w_{t,i,j} \\
     1, & {\xi}_{t,i}\geq w_{t,i,j }.
\end{array}  \right.
\end{equation}
\end{defn}

\begin{rem}
In Definition \ref{def1}, $V_{t,i,j}$ changes continuously with respect to ${\xi}_{t,i}$ and $Q_{t,i,j}$ is a binary-valued indicator of whether ${\xi}_{t,i}$ exceeds the breakpoint $w_{t,i,j}$. This lifting function allows us to design continuous and integer control policies according to the values of $V_{t,i,j}$ and $Q_{t,i,j}$, respectively.
\end{rem}

In an inverse direction, we can construct a recovery transform $\boldsymbol{R}_{t,i}$ defined as
\begin{equation}
\boldsymbol{R}_{t,i}\boldsymbol{\xi}_{t,i}^{\ast}=\sum_{j=1}^{p_{t,i}}V_{t,i,j}={\xi}_{t,i},
\end{equation}
 to map $\boldsymbol{\xi}_{t,i}^{\ast}$ back to ${\xi}_{t,i}$.

With respect to the lifted support $\varXi_{t,i}^\ast$, previous research mainly focuses on the structure of its closed convex hull, while the analysis of the lifted support itself is ignored \cite{georghiou2015generalized,bertsimas2018binary}. 
To fill this gap, we investigate the structure of the lifted support $\varXi_{t,i}^\ast$ in Proposition \ref{prop}, which is instrumental in developing the DR-LCP method to find the optimal LCP.

\begin{pro}
\label{prop}
Let $G_{t,i}(\cdot)$ and $\varXi_{t,i}^{\ast}$ be given in Definition \ref{def1}. Then,
the lifted support $\varXi_{t,i}^{\ast}$ is the union of $p_{t,i}$ line segments $K_{t,i,j}^{\ast},j=1,\cdots,p_{t,i}$ with the following analytical formulation:
\begin{equation}
\label{supp}
  \begin{gathered}
    K_{t,i,j}^{\ast}=\big\{ \left. \lambda G_{t,i}\left( w_{t,i,(j-1)} \right) +(1-\lambda) G_{t,i}^{-}\left(  w_{t,i,j} \right)\ \right|\\ \lambda\in [0,1) \big\},
     \emph{ for }  j\in [p_{t,i}-1],\\
    K_{t,i,p_{t,i}}^{\ast}=\big\{ \left. \lambda G_{t,i}\left( w_{t,i,(p_{t,i}-1)} \right) +(1-\lambda) G_{t,i}^{-}\left(  w_{t,i,p_{t,i}} \right)\right|\\ \lambda\in [0,1] \big\},
\end{gathered}  
\end{equation}
{where $G_{t,i}^{-}(  w_{t,i,j}):=\lim_{\xi\nearrow w_{t,i,j}}G_{t,i}(\xi)$ is the left limit of lifting function $G_{t,i}$ at point $ w_{t,i,j}$.}
\end{pro}
\begin{proof}
    Based on the $p_{t,i}-1$ breakpoints, the support $\varXi_{t,i}$ is partitioned into $p_{t,i}$ intervals $K_{t,i,j},j=1,\cdots,p_{t,i}$ as follows.
\begin{equation}
\begin{gathered}
    \varXi_{t,i}=\cup_{j=1}^{p_{t,i}}K_{t,i,j},\\
    K_{t,i,j}=\left[w_{t,i,(j-1)},w_{t,i,j}\right), \text{for } j\in [p_{t,i}-1],\\
    K_{t,i,p_{t,i}}=\left[w_{t,i,(p_{t,i}-1)},w_{t,i,p_{t,i}}\right].
\end{gathered}
\end{equation}

Therefore, by defining $K_{t,i,j}^\ast=G_{t,i}(K_{t,i,j})$, we have
\begin{equation}
    \varXi_{t,i}^\ast=\cup_{j=1}^{p_{t,i}}G_{t,i}(K_{t,i,j})=\cup_{j=1}^{p_{t,i}}K_{t,i,j}^\ast.
\end{equation}

Noting that the lifting function $G_{t,i}$ for dimension ${\xi}_{t,i}$ is linear on each interval $K_{t,i,j}$, so the segment expression (\ref{supp}) is naturally derived.
\end{proof}

By stacking up each lifted dimension, we can define the lifting function of the past disturbance $\boldsymbol{\xi}_{[t]}$ up to time $t$ below.
\begin{defn}\label{def_3}
    Let the lifting function $G_{t,i}$ be well-defined for each dimension of the disturbance. The lifting function $G_{[t]}$ maps the past disturbance $\boldsymbol{\xi}_{[t]}$ into lifted past disturbance $\boldsymbol{\xi}_{[t]}^\ast$ as follows.
\begin{align}
    G_{[t]}:\varXi_{[t]}\rightarrow\varXi_{[t]}^{\ast};&\
    \boldsymbol{\xi}_{[t]}{\rightarrow}G_{[t]}(\boldsymbol{\xi}_{[t]})=\boldsymbol{\xi}_{[t]}^{\ast}\nonumber\\
    :=(&G_{1,1}({\xi}_{1,1}),\cdots,G_{1,n_\xi}({\xi}_{1,n_\xi}),\nonumber\\&G_{2,1}({\xi}_{2,1}),\cdots,G_{t,n_\xi}({\xi}_{t,n_\xi}))^T\\
    =(&\boldsymbol{V}_{1,1}^T,\boldsymbol{Q}_{1,1}^T,\cdots,\boldsymbol{V}_{1,n_\xi}^T,\boldsymbol{Q}_{1,n_\xi}^T,\nonumber\\&\boldsymbol{V}_{2,1}^T,\boldsymbol{Q}_{2,1}^T,\cdots,\boldsymbol{V}_{t,n_\xi}^T,\boldsymbol{Q}_{t,n_\xi}^T)^T.\nonumber
\end{align}
\end{defn}

Therefore, the lifted past disturbance $\boldsymbol{\xi}_{[t]}^{\ast}$ can be expressed as a projection of the lifted overall disturbance $\boldsymbol{\xi}^{\ast}=G_{[\mathcal{T}]}(\boldsymbol{\xi}_{[\mathcal{T}]})$. Specifically, this projection is given by
\begin{equation} \label{proj} \boldsymbol{\xi}_{[t]}^{\ast}=\boldsymbol{P}_t\boldsymbol{\xi}^{\ast}, \end{equation}
where $\boldsymbol{P}_t=(\boldsymbol{I},\boldsymbol{0})$.

Accordingly, the recovery transform $\boldsymbol{R}$ for the lifted overall disturbance $\boldsymbol{\xi}^{\ast}$ is constructed as
\begin{equation}
    \boldsymbol{R}=\text{diag}\left( \boldsymbol{R}_{1,1},\cdots,\boldsymbol{R}_{1,n_\xi},\boldsymbol{R}_{2,1},\cdots,\boldsymbol{R}_{\mathcal{T},n_\xi} \right),
\end{equation}
and the relation 
\begin{equation}
\label{recover}
    \boldsymbol{R}\boldsymbol{\xi}^{\ast}=\boldsymbol{\xi}
\end{equation}
can be easily verified.

For the convenience of constructing the LCP, we further define the continuous and integer parts of the lifted past disturbance $\boldsymbol{\xi}_{[t]}^{\ast}$ as
\begin{gather}                  
(\boldsymbol{V}_{1,1}^T,,\cdots,\boldsymbol{V}_{1,n_\xi}^T,\boldsymbol{V}_{2,1}^T,,\cdots,\boldsymbol{V}_{t,n_\xi}^T)^T=\boldsymbol{P}_{V_t}\boldsymbol{\xi}_{[t]}^{\ast},\\
(\boldsymbol{Q}_{1,1}^T,,\cdots,\boldsymbol{Q}_{1,n_\xi}^T,\boldsymbol{Q}_{2,1}^T,,\cdots,\boldsymbol{Q}_{t,n_\xi}^T)^T=\boldsymbol{P}_{Q_t}\boldsymbol{\xi}_{[t]}^{\ast},
\end{gather}
and we denote the dimension of these two vectors as $n_{V_t}=\sum_{t'=1}^t \sum_{i=1}^{n_\xi}{p_{t',i}}$ and $n_{Q_t}=\sum_{t'=1}^t \sum_{i=1}^{n_\xi}[{p_{t',i}-1}]$.

\subsection{Continous and Integer Lifted Control Policy}
Based on the lifting function, the LCP for continuous control $\boldsymbol{u}_t^\pi$ is defined as an affine function of the continuous part of the lifted past disturbance $\boldsymbol{\xi}_{[t]}^{\ast}$ \cite{georghiou2015generalized}, \emph{i.e.},
\begin{equation}
\label{conti}
    \begin{gathered}
       { \boldsymbol{u}_{t}^\pi(\boldsymbol{\xi}_{[t]})=\boldsymbol{Y}_t\boldsymbol{P}_{V_t}G_{[t]}(\boldsymbol{\xi}_{[t]})+\boldsymbol{y}_t^0=\boldsymbol{Y}_t\boldsymbol{P}_{V_t}\boldsymbol{\xi}_{[t]}^{\ast}+\boldsymbol{y}_t^0}\\
=\boldsymbol{Y}_t\boldsymbol{P}_{V_t}\boldsymbol{P}_t\boldsymbol{\xi}^{\ast}+\boldsymbol{y}_t^0,
\end{gathered}
\end{equation}
where $\boldsymbol{Y}_t\in \mathbb{R}^{n_{u}\times n_{V_t}}$ and $\boldsymbol{y}_t^0\in \mathbb{R}^{n_{u}}$.

Similarly, the LCP for integer control $\boldsymbol{\gamma}_t^\pi$ is defined as the affine function of the integer part of $\boldsymbol{\xi}_{[t]}^{\ast}$ \cite{bertsimas2018binary}, \emph{i.e.},
\begin{equation}
\label{binar}
\begin{gathered}
    { \boldsymbol{\gamma}_{t}^\pi(\boldsymbol{\xi}_{[t]})=\boldsymbol{Z}_t\boldsymbol{P}_{Q_t} G_{[t]}(\boldsymbol{\xi}_{[t]})+\boldsymbol{z}_t^0=\boldsymbol{Z}_t\boldsymbol{P}_{Q_t}\boldsymbol{\xi}_{[t]}^{\ast}+\boldsymbol{z}_t^0}\\
=\boldsymbol{Z}_t\boldsymbol{P}_{Q_t}\boldsymbol{P}_t\boldsymbol{\xi}^{\ast}+\boldsymbol{z}_t^0,
    \end{gathered}
\end{equation}
where $\boldsymbol{Z}_t\in \mathbb{Z}^{n_{\gamma}\times n_{Q_t}}$ and $\boldsymbol{z}_t^0\in \mathbb{Z}^{n_{\gamma}}$. 

\begin{rem}
    By the construction of the LCP, if the segment number $p_{t,i}=1$ for all dimensions of the disturbance, i.e., no breakpoint being used, the LCP degenerates to the ACP. If the segment number $p_{t,i}>1$, the ACP is a subset of the LCP, and due to the nonlinearity of the lifting function, a family of nonlinear policies for the continuous and integer controls are also included in the LCP. Moreover, using more breakpoints can enrich the content of the LCP, but it will result in a heavier computational burden since optimization over a larger family of control policies is involved.
\end{rem}

{ Based on expressions ({\ref{conti}}) and ({\ref{binar}}), we can define the lifted policy space.}
{
\begin{defn}[Lifted Control Policy Space]\label{def_LCPS}
 With breakpoints $\boldsymbol{w}=(w_{t,i,j},t\in [\mathcal{T}],i\in [n_\xi],j\in [p_{t,i}-1])$, the continuous and integer lifted policy spaces $\Pi^{\boldsymbol{u}}({\boldsymbol{w}})$ and ${\Pi}^{\boldsymbol{\gamma}}({\boldsymbol{w}})$ are
    \begin{gather}
    { \begin{aligned}
        \Pi^{{u}}&({\boldsymbol{w}})=\{\pi^u=  (\boldsymbol{u}_1^\pi(\cdot),\cdots,\boldsymbol{u}_{\mathcal{T}}^\pi(\cdot)  ) |\boldsymbol{u}_t^\pi(\boldsymbol{\xi}_{[t]})\\&=\boldsymbol{Y}_t\boldsymbol{P}_{V_t}G_{[t]}(\boldsymbol{\xi}_{[t]})+\boldsymbol{y}_t^0,\boldsymbol{Y}_t\in \mathbb{R}^{n_{u}\times n_{V_t}},\boldsymbol{y}_t^0\in \mathbb{R}^{n_{u}}\},
    \end{aligned}}\\
     { \begin{aligned}
        \Pi^{{\gamma}}&({\boldsymbol{w}})=\{ \pi^{\gamma}= (\boldsymbol{\gamma}_{1}^\pi(\cdot),\cdots,\boldsymbol{\gamma}_{\mathcal{T}}^\pi(\cdot) ) | 
        \boldsymbol{\gamma}_{t}^\pi(\boldsymbol{\xi}_{[t]})\\&=\boldsymbol{Z}_t\boldsymbol{P}_{Q_t}G_{[t]}(\boldsymbol{\xi}_{[t]})+\boldsymbol{z}_t^0,\boldsymbol{Z}_t\in \mathbb{Z}^{n_{\gamma}\times n_{Q_t}},\boldsymbol{z}_t^0\in \mathbb{Z}^{n_{\gamma}}\}.
    \end{aligned}}
    \end{gather}
\end{defn}}

{ By optimizing over the LCP, the (\mbox{\ref{drsc}}) problem becomes computing $J^\ast ( \Pi^{{u}}({\boldsymbol{w}}),\Pi^{{\gamma}}({\boldsymbol{w}}))$. To understand the performance of the LCP, we conduct a theoretical analysis in the next section. Subsequently, in Section \mbox{\ref{refor}}, we develop the methodology to compute $J^\ast ( \Pi^{{u}}({\boldsymbol{w}}),\Pi^{{\gamma}}({\boldsymbol{w}}))$ and the associated best lifted controller.}

{
\section{Performance Analysis}
\label{thero_anal}
By Definition \ref{def_LCPS}, the LCP space is determined by the breakpoints $\boldsymbol{w}$, and when the number of breakpoints increases, more control policies are contained in the LCP space. Therefore, it can be expected that using more breakpoints will lead to better performance. 

In this section, we will analyze this breakpoint-performance relation and answer the following two questions:
\begin{enumerate}
    \item What is the asymptotic performance of the LCP when the number of breakpoints goes to infinity?
    \item When finitely many breakpoints are used, what is the tightest bound for the disparity between this non-asymptotic performance and the asymptotic performance?
\end{enumerate}

To quantify the effectiveness of the breakpoints in segmenting the support, we first define the largest segment $\epsilon(\boldsymbol{w})$ induced by breakpoints $\boldsymbol{w}$ as follows.
\begin{defn}\label{def_seg_larg}
    The largest segment induced by breakpoints $\boldsymbol{w}$ is defined as
    \begin{equation}
        \epsilon(\boldsymbol{w})=\max_{t\in [\mathcal{T}],i\in [n_\xi],j\in[p_{i,j}]}\{ w_{t,i,j}-w_{t,i,j-1}\}.
    \end{equation}
\end{defn}

With this notion, the asymptotic performance naturally refers to the performance when $ \epsilon(\boldsymbol{w})$ goes to 0.

We then define the distance between continuous control policies in Definition \ref{def_distance}.
\begin{defn}[Distance between continuous control policies]\label{def_distance}
    Given two continuous policies $\pi^u_1$ and $\pi^u_2$, their distance is defined as
    \begin{equation}
        d(\pi^u_1,\pi^u_2)=\sum_{t=1}^{\mathcal{T}}\sum_{j=1}^{n_u} \| {u}^{\pi_1}_{t,j}-{u}^{\pi_2}_{t,j}  \|_{\infty},
    \end{equation}
    where $\|  \cdot \|_\infty$ is the supremum norm over function space and $\boldsymbol{u}_t^\pi=(u_{t,1}^\pi,\cdots,u_{t,n_{u}}^\pi)$.
\end{defn}

\subsection{Asymptotic Performance Analysis}
\label{as_section}
To analyze the asymptotic performance, we first characterize the ``limit LCP space" when $\epsilon(\boldsymbol{w})$ goes to 0. To this end, we define additive control policy space in Definition \ref{def_add_con}, and in Proposition \ref{thm_asy}, we prove that the ``limit LCP space" coincides with this additive control policy space.

\begin{defn}[Additive Control Policy Space]\label{def_add_con}
The additive continuous control policy space $\mathfrak{N}^{u}$ and additive integer control policy space $\mathfrak{N}^{\gamma}$ are defined as
     \begin{gather}
     \mathfrak{N}^{u} = \Bigg\{ \pi^u=  (\boldsymbol{u}_1^\pi(\cdot),\cdots,\boldsymbol{u}_{\mathcal{T}}^\pi(\cdot)  )  \Bigg| \forall t\in [\mathcal{T}], \forall j \in [n_u],\nonumber\\ {u}_{t,j}^\pi(\boldsymbol{\xi}_{[t]})=\sum_{t'=1}^{t} \sum_{i=1}^{n_\xi} f_{t,j,t',i}(\xi_{t',i}), f_{t,j,t',i} \in C(\varXi_{t',i})   \Bigg\},\\
        \mathfrak{N}^{\gamma} =\Bigg\{\pi^{\gamma}= (\boldsymbol{\gamma}_{1}^\pi(\cdot),\cdots,\boldsymbol{\gamma}_{\mathcal{T}}^\pi(\cdot) )\Bigg|\forall t\in [\mathcal{T}],\forall j \in [n_\gamma], \nonumber\\{\gamma}_{t,j}^\pi(\boldsymbol{\xi}_{[t]})=\sum_{t'=1}^{t} \sum_{i=1}^{n_\xi} f_{t,j,t',i}(\xi_{t',i}), f_{t,j,t',i} \in \emph{PInt}(\varXi_{t',i})    \Bigg\},
    \end{gather}
    where $C(\varXi_{t',i})$ denotes the continuous function space and $\emph{PInt}(\varXi_{t',i})$ denotes the space of piecewise integer functions, {i.e.},
    \begin{equation*}
         \begin{gathered}
        \emph{PInt}(\varXi_{t',i})=\{ f(\cdot)| \exists {l}_{t,i} =w_{0}\leq w_{1}\leq \cdots \leq w_{n-1}\leq w_{n}= {v}_{t,i}, \\\emph{s.t. }\forall i\in[N], f(\cdot) \emph{ is an integer constant on }[w_{i-1},w_{i}) \}.
    \end{gathered}
    \end{equation*}
\end{defn}

\begin{pro}\label{thm_asy}
    Let $\Pi^u_\ast$ and $\Pi^{\gamma}_\ast$ denote the limits of the lifted control policy space, {i.e.},
    \begin{gather}
        \Pi^u_\ast=\lim_{\delta \searrow 0} \bigcup_{\boldsymbol{w}:\epsilon(\boldsymbol{w})\geq \delta} \Pi^{u}(\boldsymbol{w})=
        \bigcup_{n =1,\cdots,\infty} \bigcup_{\boldsymbol{w}:\epsilon(\boldsymbol{w})\geq 1/n} \Pi^{u}(\boldsymbol{w}),\nonumber\\
        \Pi^{\gamma}_\ast=\lim_{\delta \searrow 0} \bigcup_{\boldsymbol{w}:\epsilon(\boldsymbol{w})\geq \delta} \Pi^{\gamma}(\boldsymbol{w})=
        \bigcup_{n =1,\cdots,\infty} \bigcup_{\boldsymbol{w}:\epsilon(\boldsymbol{w})\geq 1/n} \Pi^{\gamma}(\boldsymbol{w}).\nonumber
    \end{gather}

    Then,
    \begin{enumerate}[(i)]
    \item $\Pi^{\gamma}_\ast=\mathfrak{N}^{\gamma}$\label{prop_per_1}.
    
    \item When equipped with topology induced by norm in Definition \ref{def_distance}, $\Pi^{u}_\ast$ is dense in $\mathfrak{N}^{u}$,  \emph{i.e.},
    \begin{equation}
        \Pi^{u}_\ast \subset \mathfrak{N}^{u} \subset \overline{\Pi^{u}_\ast},
    \end{equation}
    where $\overline{\Pi^{u}_\ast}$ is the closure of $\Pi^{u}_\ast$.\label{prop_per_2}
    \end{enumerate}
    
\end{pro}

\begin{proof}
    By Definition \ref{def_LCPS}, for $\pi^u=  (\boldsymbol{u}_1^\pi(\cdot),\cdots,\boldsymbol{u}_{\mathcal{T}}^\pi(\cdot)  )  \in \Pi^u(\boldsymbol{w})$ and $\pi^\gamma= (\boldsymbol{\gamma}_{1}^\pi(\cdot),\cdots,\boldsymbol{\gamma}_{\mathcal{T}}^\pi(\cdot) ) \in \Pi^{\gamma}(\boldsymbol{w})$, they satisfy
        \begin{figure}[!htp]  
\centering 
\includegraphics[width=0.4\textwidth]{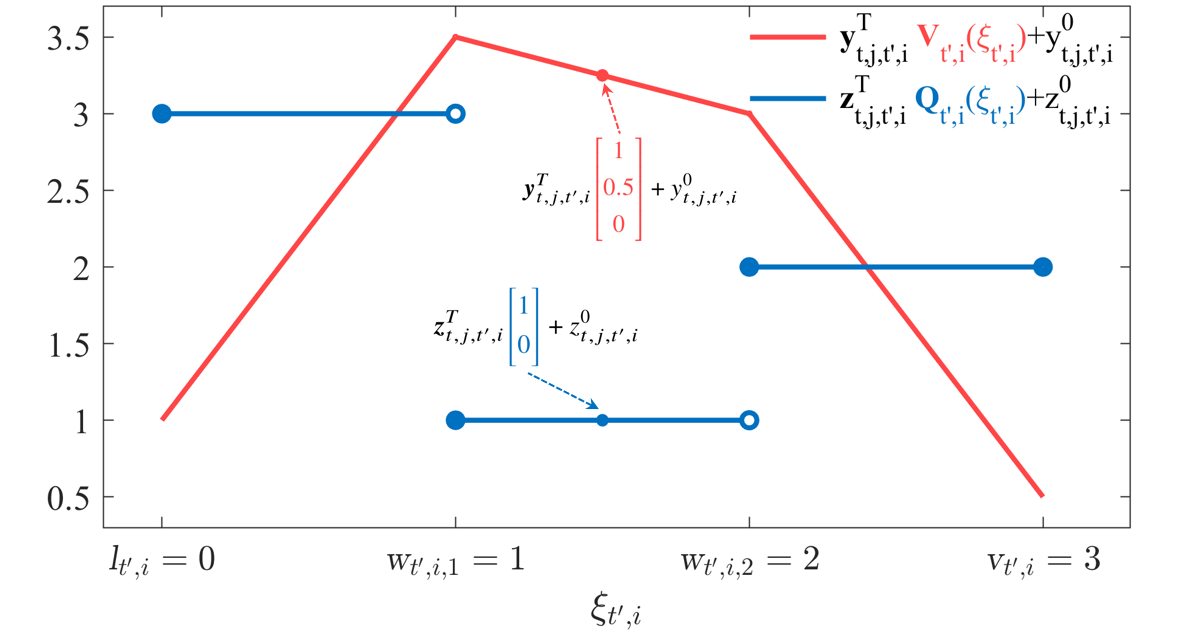} \vspace{-5pt}
{\caption{Illustration of the LCP. (In this example, the support $\varXi_{t',i}=[0,3]$ is divided by two breakpoints $w_{t',i,1}=1$ and $w_{t',i,2}=2$. As for the control policy, we take $\boldsymbol{y}_{t,j,t',i}=[2.5,-0.5,-2.5]^T,{y}_{t,j,t',i}^0=1.0$ for continuous control and $\boldsymbol{z}_{t,j,t',i}=[-2,1]^T,{z}_{t,j,t',i}^0=3$ for integer control.)}\label{Fig1}
}
\vspace{-15pt}
\end{figure}
    \begin{gather}
    \begin{gathered}
         u_{t,j}^\pi(\boldsymbol{\xi}_{[t]})=\sum_{t'=1}^{t} \sum_{i=1}^{n_\xi}\Big\{ \boldsymbol{y}_{t,j,t',i}^T \boldsymbol{V}_{t',i}(\xi_{t',i}) + y^0_{t,j,t',i} \Big\},\\\forall t\in [\mathcal{T}],\forall j \in [n_u],
    \end{gathered}\\
    \begin{gathered}
        \gamma_{t,j}^\pi(\boldsymbol{\xi}_{[t]})= \sum_{t'=1}^{t} \sum_{i=1}^{n_\xi} 
        \Big\{\boldsymbol{z}_{t,j,t',i}^T \boldsymbol{Q}_{t',i}(\xi_{t',i}) + z^0_{t,j,t',i}\Big\},\\\forall t\in [\mathcal{T}],\forall j \in [n_\gamma],
    \end{gathered}
    \end{gather}
    where $\boldsymbol{y}_{t,j,t',i}\in \mathbb{R}^{p_{t',i}}$, $y^0_{t,j,t',i}\in \mathbb{R}$, $\boldsymbol{z}_{t,j,t',i}\in \mathbb{Z}^{p_{t',i}-1}$, and $z^0_{t,j,t',i}\in \mathbb{Z}$.

    By definition of functions $\boldsymbol{V}_{t',i}$ and $\boldsymbol{Q}_{t',i}$ (in Definition \ref{def1}), $\boldsymbol{y}_{t,j,t',i}^T \boldsymbol{V}_{t',i}(\xi_{t',i}) + y^0_{t,j,t',i}$ is a piecewise linear function on $\varXi_{t',j}$ and $\boldsymbol{z}_{t,j,t',i}^T \boldsymbol{Q}_{t',i}(\xi_{t',i}) + z^0_{t,j,t',i}$ is a piecewise integer function on $\varXi_{t',j}$. To see this result more clearly, we give an illustration of a concrete example in Fig. \ref{Fig1}.

        Therefore, lifted control policies are also additive control policies, so
    \begin{gather}
        \Pi^{\gamma}_\ast\subset\mathfrak{N}^{\gamma},\Pi^{u}_\ast \subset \mathfrak{N}^{u}.
    \end{gather}

    By adjusting the breakpoints $\boldsymbol{w}$, $\boldsymbol{z}_{t,j,t',i}^T \boldsymbol{Q}_{t',i}(\xi_{t',i}) + z^0_{t,j}$ can represent any piecewise integer function on $\varXi_{t',j}$. 
    
    Therefore, (\ref{prop_per_1}) holds. 

    For (\ref{prop_per_2}), $\forall \pi^u_1 \in \mathfrak{N}^{u}$ and $\forall \pi^u_2 \in \Pi^u(\boldsymbol{w})$, we have 
    \begin{gather*}
        d(\pi^u_1 ,\pi^u_2)= \sum_{t=1}^{\mathcal{T}}\sum_{j=1}^{n_u} \| {u}^{\pi_1}_{t,j}-{u}^{\pi_2}_{t,j}  \|_{\infty} = \\\sum_{t=1}^{\mathcal{T}}\sum_{j=1}^{n_u} \left|\sum_{t'=1}^{t} \sum_{i=1}^{n_\xi}\Big\{  f_{t,j,t',i}(\xi_{t',i}) - \boldsymbol{y}_{t,j,t',i}^T \boldsymbol{V}_{t',i}(\xi_{t',i}) - y^0_{t,j,t',i} \Big\} \right|\\
        \leq \sum_{t=1}^{\mathcal{T}}\sum_{j=1}^{n_u} \sum_{t'=1}^{t} \sum_{i=1}^{n_\xi}\left|   f_{t,j,t',i}(\xi_{t',i}) - \boldsymbol{y}_{t,j,t',i}^T \boldsymbol{V}_{t',i}(\xi_{t',i}) - y^0_{t,j,t',i}  \right|,
    \end{gather*}
    where the inner term $\|   f_{t,j,t',i}(\xi_{t',i}) - \boldsymbol{y}_{t,j,t',i}^T \boldsymbol{V}_{t',i}(\xi_{t',i}) - y^0_{t,j,t',i} \|_{\infty}$ measures the distance between a continuous function $f_{t,j,t',i}(\xi_{t',i})$ and a piecewise linear function $\boldsymbol{y}_{t,j,t',i}^T \boldsymbol{V}_{t',i}(\xi_{t',i}) + y^0_{t,j,t',i}$.
    
    It is well-known that when the length of each segment is small enough ($\epsilon(\boldsymbol{w})\to 0$),  piecewise linear functions can approximate a continuous function to arbitrary accuracy under the supremum norm, so 
    \begin{gather}
       \forall \pi^u_1 \in \mathfrak{N}^{u}, \lim_{\epsilon(\boldsymbol{w})\searrow 0} \min_{\pi^u_2 \in \Pi^u(\boldsymbol{w})} d(\pi^u_1 ,\pi^u_2)=0.
    \end{gather}

    Therefore, (\ref{prop_per_2}) holds.
\end{proof}

Based on Proposition \ref{thm_asy}, we prove the asymptotic performance theorem in Theorem \ref{thm_as_per}, which states that the asymptotic performance given by LCP when $\epsilon(\boldsymbol{w})\to 0$ coincides with the performance of the additive control policy space.

\begin{thm}[Asymptotic Performance]\label{thm_as_per}
    If cost functions $c_t(\boldsymbol{x}_t,\boldsymbol{u}_t,\boldsymbol{\gamma}_t)$ are absolutely continuous, and for each feasible integer controller ${\pi}^\gamma$, the set of feasible continuous controllers
    \begin{equation}\label{thm_set_feas}
       \mathfrak{A}({\pi}^\gamma) = \{ {\pi}^u \in  \mathfrak{N}^{u} | \boldsymbol{\pi}= ({\pi}^u,{\pi}^\gamma) \text{ is feasible for (\ref{drsc})}  \}\nonumber
    \end{equation}
    has non-empty interior, then 
    \begin{equation}
        \lim_{\delta \searrow 0} \min_{w:\epsilon(\boldsymbol{w})\geq \delta} J^{\ast}(\Pi^u(\boldsymbol{w}),\Pi^\gamma(\boldsymbol{w}))= J^\ast (\mathfrak{N}^{u},\mathfrak{N}^{\gamma} ).
    \end{equation}
\end{thm}
\begin{proof}
    According to Proposition \ref{thm_asy}.(\ref{prop_per_1}), for any $(\Bar{\pi}^u,\Bar{\pi}^\gamma)\in \mathfrak{N}^u \times \mathfrak{N}^\gamma$ feasible for problem (\ref{drsc}), there exists a $\boldsymbol{w}$ such that $\Bar{\pi}^\gamma \in \Pi^\gamma(\boldsymbol{w})$.
    
    Therefore, by the absolute continuity of the cost functions $c_t$, it remains to prove that there exists a sequence of continuous controllers $\pi^u_n\in \Pi^u_\ast$ such that $(\pi^u_n,\Bar{\pi}^\gamma)$ are feasible for problem (\ref{drsc}) and satisfy $\lim_{n \to \infty} d( \pi^u_n , \Bar{\pi}^u )=0$.

   Since the set of feasible continuous controllers $\mathfrak{A}(\Bar{\pi}^\gamma)$ has non-empty interior, there exists a $\pi^u_0\in \mathfrak{N}^{u}$ such that $\{ \pi^u\in   \mathfrak{N}^{u} | d(\pi^u,\pi^u_0  ) < \rho \}\subset \mathfrak{A}(\Bar{\pi}^\gamma)$. Let $d_0=d( \Bar{\pi}^u ,\pi_0^u )$. Consider the set 
    \begin{equation}
        \mathcal{B}_n=\Bigg\{ \pi^u\in   \mathfrak{N}^{u} \Bigg|\ d \Bigg( \pi^u , \frac{2^n-1}{2^n} \Bar{\pi}^u+\frac{1}{2^n} \pi^u_0\Bigg) < \frac{\rho}{2^n}  \Bigg\}.
    \end{equation}
    
    Since $\mathcal{B}_n$ is open in $\mathfrak{N}^{u}$ and $\Pi^{u}_\ast$ is dense in $\mathfrak{N}^{u}$ by Proposition \ref{thm_asy}.(\ref{prop_per_2}), then $\Pi^{u}_\ast\cap \mathcal{B}_n$ is dense in $\mathfrak{N}^{u}\cap \mathcal{B}_n$. Therefore, we can choose $\pi^u_n$ from $\mathcal{B}_n$ such that $\pi^u_n\in \Pi^{u}_\ast$.

    By this selection, we have
    \begin{gather}
    \begin{gathered}
        d(\pi^u_n, \Bar{\pi}^u )\\\leq d \Bigg( \pi^u_n , \frac{2^n-1}{2^n} \Bar{\pi}^u+\frac{1}{2^n} \pi^u_0\Bigg) + d\Bigg( \frac{2^n-1}{2^n} \Bar{\pi}^u+\frac{1}{2^n} \pi^u_0, \Bar{\pi}^u \Bigg)\\
        <\frac{\rho}{2^n} + d\Bigg(\frac{1}{2^n} \pi^u_0,\frac{1}{2^n} \Bar{\pi}^u \Bigg)=\frac{\rho}{2^n}+\frac{1}{2^n}d(\pi^u_0,\Bar{\pi}^u)=\frac{\rho+d_0}{2^n}.
    \end{gathered}
    \end{gather}
    
    Therefore, the second condition $\lim_{n \to \infty} d( \pi^u_n , \Bar{\pi}^u )=0$ is satisfied.

    To prove that $\pi^u_n\in \mathfrak{A}({\pi}^\gamma)$, consider $\hat{\pi}_u=2^n\pi^u_n -(2^n-1)\Bar{\pi}^u\in \mathfrak{N}^{u}$. Notice that 
    \begin{gather}
        d(\hat{\pi}_u,{\pi}^u_0 )
        =d(2^n\pi^u_n-(2^n-1)\Bar{\pi}^u,{\pi}^u_0 )\\=d(2^n\pi^u_n ,(2^n-1)\Bar{\pi}^u+ {\pi}^u_0 )\\
        =2^n d( \pi^u_n ,(2^n-1)/2^n\Bar{\pi}^u+1/2^n {\pi}^u_0 )< 2^n \times \rho/2^n =\rho.
    \end{gather}
    
    Therefore, $\hat{\pi}_u\in \{ \pi^u\in   \mathfrak{N}^{u} | d(\pi^u, \pi^u_0  ) < \rho \}\subset \mathfrak{A}(\Bar{\pi}^\gamma)$.

    By the linearity of the constraints of the problem (\ref{drsc}), $\mathfrak{A}(\Bar{\pi}^\gamma)$ is convex. Therefore, $\pi^u_n\in \mathfrak{A}(\Bar{\pi}^\gamma)$ since $\pi^u_n=1/2^n \hat{\pi}^u +(2^n-1)/2^n \Bar{\pi}^u$.
\end{proof}
\vspace{-10pt}
\subsection{Non-asymptotic Performance Analysis}
\label{no_as_section}
In this part, we aim to analyze the non-asymptotic performance of LCP induced by particular breakpoints $\boldsymbol{w}$.
Due to the non-convex nature of the integer lifted policy space $\Pi^\gamma (\boldsymbol{w})$, conducting non-asymptotic analysis with integer control poses considerable challenges. Hence, in this section, we consider systems with only continuous control and analyze the non-asymptotic performance with Lipschitz continuous controllers which are defined as follows.
\begin{defn}[Lipschitz Continuous Control Policy Space]\label{lip_conti_cont}
    Let the disturbance space $\varXi$ be equipped with the 1-norm distance. The $L_\pi$ Lipschitz continuous LCP space is defined as
    \begin{gather}
        \Pi^u_{L_\pi}(\boldsymbol{w})=\{ \pi^u \in \Pi^u(\boldsymbol{w})|u_{t,j}^\pi \text{ is } L_\pi \text{ Lipschitz}\},
    \end{gather}
    and the $L_\pi$ Lipschitz continuous additive control policy space is defined as
    \begin{gather}
        \mathfrak{N}^{u}_{L_\pi}=\{ \pi^u \in \mathfrak{N}^{u}|u_{t,j}^\pi \text{ is } L_\pi \text{ Lipschitz}\}.
    \end{gather}
\end{defn}

By Definition \ref{lip_conti_cont}, it is noteworthy that for an additive control policy $u_{t,j}^\pi=\sum_{t'=1}^t\sum_{i=1}^{n_\xi}u_{t,j,t',i}^\pi$, if $u_{t,j}^\pi$ is $L_\pi$ Lipschitz, each $u_{t,j,t',i}^\pi$ is also $L_\pi$ Lipschitz.

Constraint (\ref{hard}) that defines the feasible controllers also poses challenges to non-asymptotic analysis, and to handle this feasibility issue, we simplify the feasible policy space to a hyper-ball centered at a given control policy.

Under the above simplifications, we establish the tightest non-asymptotic performance bound in the following theorem.

\begin{thm}[Non-asymptotic Performance Bound]\label{Non_asymptotic}
Let the system only involve continuous control and the feasible policy space be a hyper ball $ \mathfrak{C}$ with radius $\rho$ centered at a given lifted continuous control policy $\pi_0^u\in \Pi^u(\boldsymbol{w})$, {i.e.},
\begin{equation}
    \mathfrak{C}=\{ \pi^u | d(\pi^u,\pi_0^u)\leq \rho\}.
\end{equation}

If cost functions $c_t(\boldsymbol{x}_t,\boldsymbol{u}_t)$ are Lipschitz continuous with Lipschitz constant $L_c$, then
    \begin{equation}
         J^\ast (\Pi^u_{L_\pi}(\boldsymbol{w}) )  -  J^\ast (\mathfrak{N}^{u}_{L_\pi} ) \lesssim \mathcal{O}\left(\epsilon(\boldsymbol{w}) \right),
    \end{equation}
    and this bound is tight.

As a consequence, if the breakpoints are selected to segment each dimension of the support into $n$ segments of equal length, then $\epsilon(\boldsymbol{w})\lesssim \mathcal{O}\left(1/n \right)$, and as a result
    \begin{equation}
         J^\ast (\Pi^u_{L_\pi}(\boldsymbol{w}) )  -  J^\ast (\mathfrak{N}^{u}_{L_\pi} ) \lesssim \mathcal{O}\left(1/n \right).
    \end{equation}
\end{thm}

\begin{proof} 
    Notice that
    \begin{gather}
    J^\ast (\Pi^u_{L_\pi}(\boldsymbol{w}) )  -  J^\ast (\mathfrak{N}^{u}_{L_\pi} )\\
    \begin{gathered}
     =
        \min_{\pi_2\in \Pi^u_{L_\pi}(\boldsymbol{w})\cap \mathfrak{C}}\ \max_{\mathbb{P}\in \mathscr{B}}\mathbb{E}_{\mathbb{P}}\left[ \sum_{t=1}^{\mathcal{T} } c_t(\boldsymbol{x}_t^{\pi_2},\boldsymbol{u}_t^{\pi_2})\right] \\- \min_{\pi_1\in \mathfrak{N}^{u}_{L_\pi}\cap \mathfrak{C}}\ \max_{\mathbb{P}\in \mathscr{B}}\mathbb{E}_{\mathbb{P}}\left[ \sum_{t=1}^{\mathcal{T} } c_t(\boldsymbol{x}_t^{\pi_1},\boldsymbol{u}_t^{\pi_1})\right]
    \end{gathered}\\
    \begin{gathered}
        =\max_{\pi_1\in \mathfrak{N}^{u}_{L_\pi}\cap \mathfrak{C}}\min_{\pi_2\in \Pi^u_{L_\pi}(\boldsymbol{w})\cap \mathfrak{C}} \Bigg\{ \max_{\mathbb{P}\in \mathscr{B}}\mathbb{E}_{\mathbb{P}}\left[ \sum_{t=1}^{\mathcal{T} } c_t(\boldsymbol{x}_t^{\pi_2},\boldsymbol{u}_t^{\pi_2})\right]\\ - \max_{\mathbb{P}\in \mathscr{B}}\mathbb{E}_{\mathbb{P}}\left[ \sum_{t=1}^{\mathcal{T} } c_t(\boldsymbol{x}_t^{\pi_1},\boldsymbol{u}_t^{\pi_1})\right]  \Bigg\}
    \end{gathered}\label{thm_min_eq}\\
    \begin{gathered}
       \leq \max_{\pi_1\in \mathfrak{N}^{u}_{L_\pi}\cap \mathfrak{C}}\ \min_{\pi_2\in \Pi^u_{L_\pi}(\boldsymbol{w})\cap \mathfrak{C}}\ \max_{\mathbb{P}\in \mathscr{B}}\\\mathbb{E}_{\mathbb{P}} \left\{ \sum_{t=1}^{\mathcal{T} } \left[c_t(\boldsymbol{x}_t^{\pi_2},\boldsymbol{u}_t^{\pi_2}) -c_t(\boldsymbol{x}_t^{\pi_1},\boldsymbol{u}_t^{\pi_1}) \right] \right\},
    \end{gathered}\label{thm_min_ep} 
\end{gather}
where $\boldsymbol{x}_t^{\pi}$ and $\boldsymbol{u}_t^{\pi}$ represent state and controller induced by $\pi$.

Inequality (\ref{thm_min_ep}) is in the form of a max-min problem or a two-player game. The first player chooses a $\pi_1\in  \mathfrak{N}^{u}_{L_\pi}\cap \mathfrak{C}$, and then in response to $\pi_1$, the second player chooses a $\pi_2\in \Pi^u_{L_\pi}(\boldsymbol{w})\cap \mathfrak{C}$. 
By Proposition \ref{thm_asy}, both $\boldsymbol{u}_t^{\pi_1}$ and $\boldsymbol{u}_t^{\pi_2}$ are in the form of additive controllers, \emph{i.e.}, $u_{t,j}^{\pi_1}=\sum_{t'\in[t],i\in [n_\xi]}u_{t,j,t',i}^{\pi_1}$ and $u_{t,j}^{\pi_2}=\sum_{t'\in[t],i\in [n_\xi]}u_{t,j,t',i}^{\pi_2}$.
Given a control policy $\pi_1$ by the first player, we set the strategy of the second player as
\begin{equation}\label{player_two_stra}
\begin{gathered}
    u_{t,j,t',i}^{\pi_2}(w_{t',i,k-1})= u_{t,j,t',i}^{\pi_1}(w_{t',i,k-1}).
\end{gathered}
\end{equation}

Equation (\ref{player_two_stra}) simply means that we construct $\pi_2$ such that it equals to $\pi_1$ at all the breakpoints $\boldsymbol{w}$, and this requirement can be satisfied since each $u_{t,j,t',i}^{\pi_2}$ is a piecewise linear function and breakpoints are just ${w}_{t',i,k-1},k\in [p_{t',i}+1]$.

We next prove that the constructed $\pi_2$ is in $\Pi^u_{L_\pi}(\boldsymbol{w})\cap \mathfrak{C}$. Notice that according to Definition \ref{def_distance},
\begin{gather}
    d(\pi^u_2,\pi^u_0)=\sum_{t=1}^{\mathcal{T}}\sum_{j=1}^{n_u} \| {u}^{\pi_2}_{t,j}-{u}^{\pi_0}_{t,j}  \|_{\infty} \\
    \leq \sum_{t=1}^{\mathcal{T}}\sum_{j=1}^{n_u}
    \sum_{t'=1}^{t}\sum_{i=1}^{n_\xi} \| ({u}^{\pi_2}_{t,j,t',i}-{u}^{\pi_0}_{t,j,t',i})\|_\infty\\
    = \sum_{t=1}^{\mathcal{T}}\sum_{j=1}^{n_u}
    \sum_{t'=1}^{t}\sum_{i=1}^{n_\xi} \max_{k\in[p_{t',i}+1]} \Big| ({u}^{\pi_2}_{t,j,t',i}-{u}^{\pi_0}_{t,j,t',i})(w_{t',i,k-1})  \Big| \label{adadafaf}\\
    =\sum_{t=1}^{\mathcal{T}}\sum_{j=1}^{n_u}
    \sum_{t'=1}^{t}\sum_{i=1}^{n_\xi}  \max_{k\in[p_{t',i}+1]} \Big| ({u}^{\pi_1}_{t,j,t',i}-{u}^{\pi_0}_{t,j,t',i})(w_{t',i,k-1})  \Big|\\
    \leq d(\pi^u_1,\pi^u_0) \leq \rho,
\end{gather}
where equality (\ref{adadafaf}) holds because both ${u}^{\pi_2}_{t,j,t',i}$ and ${u}^{\pi_0}_{t,j,t',i}$ are piecewise linear functions with breakpoints $w_{t',j,k-1}$.

To verify the Lipschitz property of $\pi_2$, it is easy to see that 
\begin{gather}
    \max_{\boldsymbol{\xi}^1,\boldsymbol{\xi}^2\in \varXi_{[t]}} \Bigg| \frac{u_{t,j}^{\pi_2}(\boldsymbol{\xi}^1)-u_{t,j}^{\pi_2}(\boldsymbol{\xi}^2)}{\|  \boldsymbol{\xi}^1-\boldsymbol{\xi}^2\|_1}\Bigg|\\
    =\max_{\boldsymbol{\xi}^1,\boldsymbol{\xi}^2\in \varXi_{[t]}}\Bigg|\frac{\sum_{t'=1}^t\sum_{i=1}^{n_\xi } u_{t,j,t',i}^{\pi_2}({\xi}^1_{t',i})-u_{t,j,t',i}^{\pi_2}({\xi}^2_{t',i})}{\sum_{t'=1}^t\sum_{i=1}^{n_\xi } |{\xi}^1_{t',i}-{\xi}^2_{t',i}  |}\Bigg|
\end{gather}
reaches the maximum when $\boldsymbol{\xi}^1,\boldsymbol{\xi}^2$ are at the breakpoints. Therefore, the Lipschitz constant of $\pi_2$ is no larger than that of $\pi_1$. 

Under this construction, the difference between $\boldsymbol{u}_{t}^{\pi_1}$ and $\boldsymbol{u}_{t}^{\pi_2}$ can be bounded as follows.
\begin{gather*}
    | u_{t,j}^{\pi_1}(\boldsymbol{\xi}) -u_{t,j}^{\pi_2}(\boldsymbol{\xi}) |\leq \sum_{t'=1}^t \sum_{i=1}^{n_\xi} | u_{t,j,t',i}^{\pi_1}({\xi}_{t',i}) -u_{t,j,t',i}^{\pi_2}({\xi}_{t',i}) |\\
    \leq \sum_{t'=1}^t \sum_{i=1}^{n_\xi}\Big\{ | u_{t,j,t',i}^{\pi_1}({\xi}_{t',i}) - u_{t,j,t',i}^{\pi_1}(\hat{w}_{t',i}) | \\+ | u_{t,j,t',i}^{\pi_2}(\hat{w}_{t',i}) -u_{t,j,t',i}^{\pi_2}({\xi}_{t',i})|\Big\} \\
    \leq \sum_{t'=1}^t \sum_{i=1}^{n_\xi} 2 L_\pi |\hat{w}_{t',i}-{\xi}_{t',i}|\leq \sum_{t'=1}^t \sum_{i=1}^{n_\xi} 2 L_\pi \epsilon(\boldsymbol{w})\lesssim \mathcal{O}(\epsilon(\boldsymbol{w})),
\end{gather*}
where $\hat{w}_{t',i}$ is the breakpoint nearest to ${\xi}_{t',i}$, \emph{i.e.}, $\hat{w}_{t',i}=\argmin_{w\in \{w_{t',i,k-1},k\in [p_{t',i}+1]\}} | w- \xi_{t',i} | $.

To bound the the difference between $\boldsymbol{x}_{t}^{\pi_1}$ and $\boldsymbol{x}_{t}^{\pi_2}$, note that 
\begin{gather*}
     \| \boldsymbol{x}_{t}^{\pi_1}(\boldsymbol{\xi}) -\boldsymbol{x}_{t}^{\pi_2}(\boldsymbol{\xi}) \|_1=\|  \boldsymbol{A}_{t}(\boldsymbol{x}^{\pi_1}_{t-1}-\boldsymbol{x}^{\pi_2}_{t-1})+\boldsymbol{B}_{t}(\boldsymbol{u}^{\pi_1}_{t-1}-\boldsymbol{u}^{\pi_2}_{t-1}) \|_1
\end{gather*}
Therefore, by induction, the difference between $\boldsymbol{x}_{t}^{\pi_1}$ and $\boldsymbol{x}_{t}^{\pi_2}$ is also in the order of $\mathcal{O}(\epsilon(\boldsymbol{w}))$.

Since $c_t$ is Lipschitz continuous, we have 
\begin{gather*}
    | c_t(x_t^{\pi_1},u_t^{\pi_1}) -c_t(x_t^{\pi_2},u_t^{\pi_2}) | \leq L_c \left( \left\|  x_t^{\pi_1}-x_t^{\pi_2} \right\| + \left\|  u_t^{\pi_1}-u_t^{\pi_2} \right\| \right)\\
    =L_c \Bigg(\sum_{j=1}^{n_x}|x_{t,j}^{\pi_1}-x_{t,j}^{\pi_2}|+ \sum_{j=1}^{n_u}|u_{t,j}^{\pi_1}-u_{t,j}^{\pi_2}|\Bigg)\lesssim\mathcal{O}(\epsilon(\boldsymbol{w})).
\end{gather*}

By the above observations, (\ref{thm_min_ep}) can be further bounded by
\begin{gather}
    (\ref{thm_min_ep}) \lesssim \max_{\pi_1\in \mathfrak{N}^{u}_{L_\pi}\cap \mathfrak{C}}\ \max_{\mathbb{P}\in \mathscr{B}}\mathbb{E}_{\mathbb{P}} [\mathcal{O}(\epsilon(\boldsymbol{w}))] \lesssim \mathcal{O}(\epsilon(\boldsymbol{w})).
\end{gather}

To prove the tightness of this bound, we need to construct a concrete example whose non-asymptotic performance is lower bounded by $\Omega(\epsilon(\boldsymbol{w}))$. To this aim, we consider a single-horizon problem where the state $x$, continuous control $u$, and disturbance $\xi$ are of dimension one. Support $\varXi=[0,1]$ and the breakpoints $\boldsymbol{w}$ segment the support into equal lengths of $\epsilon$, so $\epsilon(\boldsymbol{w})=\epsilon$. The given control policy is set to $\pi_0^u(\xi)=0,\forall \xi\in \varXi$. The ambiguity set $\mathscr{B}$ contains only one distribution $\hat{P}=1/4 \delta_{0}+1/4 \delta_{\epsilon}+1/2 \delta_{\epsilon/2}$, \emph{i.e.}, $\mathscr{B}=\{ \hat{P} \}$, which corresponds to the Wasserstein ambiguity set centered at $ \hat{P}$ with radius $0$. The state transition function is given by $x=\xi$, and the cost function is given by
\begin{equation}
    c(x,u)=\frac{L_c}{L_{\pi}+1}|u-L_{\pi}|x-\epsilon/2||,
\end{equation}
whose Lipschitz constant is 
\begin{equation}
   \frac{L_c}{L_{\pi}+1} \max \{L_{\pi},1\},
\end{equation}
which is less than $L_c$.

The optimal control $u_{\text{opt}}$ of this problem is a $L_{\pi}$ Lipschitz piecewise linear one as follows.
\begin{gather}
    u_{\text{opt}}(\xi)=\left\{  \begin{aligned}
          L_{\pi}|\xi-\epsilon/2|, &\ \xi\in[0,\epsilon]\\
        L_{\pi}\epsilon/2\hspace{20pt},&\ \xi\in (\epsilon,1].
    \end{aligned}\right.
\end{gather}

Since $u_{\text{opt}}\in \mathfrak{N}^{u}_{L_\pi}$, we have $J^\ast (\mathfrak{N}^{u}_{L_\pi} )=0$.

For controller $u \in \Pi^u_{L_{\pi}}(\boldsymbol{w})$, only the first segment $[0,\epsilon]$ affects the performance, so we have
\begin{gather}
    J^\ast (\Pi^u_{L_\pi}(\boldsymbol{w}) )-J^\ast (\mathfrak{N}^{u}_{L_\pi} )\\
   =J^\ast (\Pi^u_{L_\pi}(\boldsymbol{w}) )=\max_{u(0),u(\epsilon)}\mathbb{E}_{\hat{P}}[c(x,u)]\label{appen_pi_1}\\
   \begin{gathered}
       =\frac{L_c}{L_{\pi}+1}\max_{u(0),u(\epsilon)} \Bigg\{ \frac{1}{4}\Big|u(0)-\frac{\epsilon L_{\pi}}{2}\Big|\\+\frac{1}{4}\Big|u(\epsilon)-\frac{\epsilon L_{\pi}}{2}\Big|+\frac{1}{2}\Big| \frac{u(0)+u(\epsilon)}{2}\Big| \Bigg\}
   \end{gathered}
    \\
    \begin{gathered}
        \geq \frac{L_c}{L_{\pi}+1} \max_{u(0),u(\epsilon)} \Bigg\{  \frac{1}{4}\Big| u(0)-\frac{\epsilon L_{\pi}}{2} +u(\epsilon)-\frac{\epsilon L_{\pi}}{2} \\-2   \frac{u(0)+u(\epsilon)}{2}\Big| \Bigg\}=\frac{L_c}{L_{\pi}+1}\frac{L_{\pi}}{4} \epsilon  \gtrsim \Omega(\epsilon),
    \end{gathered}
\end{gather}
where (\ref{appen_pi_1}) holds because $u$ is linear on $[0,\epsilon]$, so it is determined by the endpoint value $u(0)$ and $u(\epsilon)$.

Therefore, the $\mathcal{O}(\epsilon(\boldsymbol{w}))$ bound is tight.
\end{proof}

By Theorem \ref{Non_asymptotic}, the tightest non-asymptotic bound is of order $1/n$, which provides guidance in determining the number of breakpoints in practice.}

\section{Solution Methodology} 
\label{refor}

{ In this section, we develop solution methodology to compute $J^\ast ( \Pi^{{u}}(\boldsymbol{w}),\Pi^{{\gamma}}(\boldsymbol{w}))$ and the associated best lifted controller.

We make the following assumption on the cost function  $c_t(\boldsymbol{x}_t,\boldsymbol{u}_{t}^\pi,\boldsymbol{\gamma}_t^\pi)$, which allows us to leverage standard techniques to handle the distributionally robust objective ({\ref{cost}}) with Wasserstein-type ambiguity sets.}

{
\begin{asum}\label{piece_assum}
     The cost function  $c_t(\boldsymbol{x}_t,\boldsymbol{u}_{t}^\pi,\boldsymbol{\gamma}_t^\pi)$ is a convex piecewise linear cost as follows.
\begin{equation}
 { c_t(\boldsymbol{x}_t,\boldsymbol{u}_{t}^\pi,\boldsymbol{\gamma}_t^\pi)=\max_{\tau\in [I]} \bigg\{ \boldsymbol{a}_{\tau}^T\boldsymbol{x}_t+\boldsymbol{b}_{\tau}^T\boldsymbol{u}_t^\pi+\boldsymbol{c}_{\tau}^T\boldsymbol{\gamma}_t^\pi \bigg\}.}
\end{equation}
\end{asum}}

It is noteworthy that the convex piecewise linear function $c_t$ can approximate a quadratic cost, which is presumed in \cite{kim2023distributional}.

According to Definition {\ref{def_LCPS}}, LCP is affine in the lifted disturbance $G(\boldsymbol{\xi})$.
By substituting the LCP formulation (\ref{conti}) and (\ref{binar}) into the system dynamics (\ref{sys}), the system state $\boldsymbol{x}_t$ is also affine in the lifted disturbance $G(\boldsymbol{\xi})$. Therefore, constraint (\ref{hard}) can be written in the following concise form.
\begin{equation}
\label{comp1}
    { \boldsymbol{m}\geq \boldsymbol{E}^TG(\boldsymbol{\xi}),\forall \boldsymbol{\xi}\in \varXi,}
\end{equation}
{ where $\boldsymbol{m}$ and $\boldsymbol{E}$ are affine functions of the LCP coefficients $\boldsymbol{Y}_t,\boldsymbol{y}_t^0,\boldsymbol{Z}_t$, and $\boldsymbol{z}_t^0$.}

Similarly, the convex piecewise linear cost  $c_t(\boldsymbol{x}_t,\boldsymbol{u}_{t}^\pi,\boldsymbol{\gamma}_t^\pi) $ is convex piecewise linear in $G(\boldsymbol{\xi})$, and since the sum of convex piecewise linear functions is also convex piecewise linear, the total cost $\sum_{t=1}^{\mathcal{T}} \alpha^tc(\boldsymbol{x}_t,\boldsymbol{u}_{t},\boldsymbol{\gamma}_t)$ can be compactly formulated as
\begin{gather}
\label{eqqq1}
    \sum_{t=1}^{\mathcal{T}} \alpha^tc_t(\boldsymbol{x}_t,\boldsymbol{u}_{t}^\pi,\boldsymbol{\gamma}_t^\pi) =\max_{k\in [K]}\bigg\{{  \boldsymbol{d}_k^TG(\boldsymbol{\xi})}+r_k \bigg\},
\end{gather}
{ where $\boldsymbol{d}_k$ and $r_k$ are affine functions of the LCP coefficients $\boldsymbol{Y}_t,\boldsymbol{y}_t^0,\boldsymbol{Z}_t$, and $\boldsymbol{z}_t^0$.}

Therefore, computing $J^\ast ( \Pi^{{u}}(\boldsymbol{w}),\Pi^{{\gamma}}(\boldsymbol{w}))$ amounts to solving the following optimization problem.
\begin{gather}
    \min_{\boldsymbol{Y}_t,\boldsymbol{y}_t^0,\boldsymbol{Z}_t,\boldsymbol{z}_t^0}\hspace{8pt}\max_{\mathbb{P}\in \mathscr{B}}\mathbb{E}_{\boldsymbol{\xi}\sim\mathbb{P}}\bigg[ \max_{k\in [K]}\bigg\{ \boldsymbol{d}_k^TG(\boldsymbol{\xi})+r_k \bigg\} \bigg]\label{obj2}\\
    \text{s.t.}\hspace{30pt}\boldsymbol{m}\geq \boldsymbol{E}^TG(\boldsymbol{\xi}),\forall \boldsymbol{\xi}\in \varXi\hspace{40pt}\label{cons2}\\
    \boldsymbol{Y}_t\in \mathbb{R}^{n_{u}\times n_{V_t}},\boldsymbol{y}_t^0\in \mathbb{R}^{n_{u}},\boldsymbol{Z}_t\in \mathbb{Z}^{n_{\gamma}\times n_{Q_t}},\boldsymbol{z}_t^0\in \mathbb{Z}^{n_{\gamma}},\label{consss2}
\end{gather}
{ where $\boldsymbol{m}$, $\boldsymbol{E}$, $\boldsymbol{d}_k$, and $r_k$ are affine in $\boldsymbol{Y}_t,\boldsymbol{y}_t^0,\boldsymbol{Z}_t$, and $\boldsymbol{z}_t^0$.}

To solve this problem, we need to develop equivalent reformulations of the objective function (\ref{obj2}) and constraint (\ref{cons2}). 
{Concerning the reformulation of the objective function (\ref{obj2}), 
although the seminar work \cite{mohajerin2018data} and follow-up papers have established standard techniques to equivalently reformulate the min-max objective (\ref{obj2}) into a single-level minimization problem, this minimization problem is a middle form since it still involves the nonlinear lifting function $G$ which is intractable to optimization solvers. 
Therefore, to derive a tractable equivalent reformulation, we need to develop reformulation techniques regarding the lifting function to further reformulate this middle form.
Instead of developing such techniques,} the previous work \cite{feng2021multistagea} relaxed the Wasserstein ambiguity set to a lifted ambiguity set to avoid this difficulty, thereby leading to a sub-optimal LCP solution. 

{In Section \ref{equi_obj_sec}, by leveraging Proposition \ref{prop}, we develop such reformulation techniques regarding the lifting function in Lemma \ref{lemma}. Based on Lemma \ref{lemma}, we derive equivalent reformulations of the objective function (\ref{obj2}) for the Wasserstein ambiguity set and its variants. Subsequently, still based on Lemma \ref{lemma}, we reformulate constraint (\ref{cons2}) in Section \ref{equi_cons_sec}.}

\subsection{Equivalent Reformulation of Objective Function} 
\label{equi_obj_sec}

We first develop the following lemma as the cornerstone of the reformulation.
\begin{lem}
\label{lemma}
Let function { $f(\boldsymbol{\xi})$} be convex piecewise linear in $G(\boldsymbol{\xi})$, i.e.,
\begin{equation}\label{29_ori}
    { f(\boldsymbol{\xi})=\max_{m\in [M]} \bigg\{ \boldsymbol{g}_m^T G(\boldsymbol{\xi}) +h_m\bigg\}.}
\end{equation}
Then, constraint (\ref{maxi}) with respect to a given point $\widehat{\boldsymbol{\xi}}$
\begin{equation}
\label{maxi}
    \eta \geq \max_{\boldsymbol{\xi}\in \varXi}\bigg\{ f(\boldsymbol{\xi})-\lambda \left\| \boldsymbol{\xi}-\widehat{\boldsymbol{\xi}} \right\|\bigg\} 
\end{equation}
can be reformulated as: $\exists\ \eta^s_{m,t,i}\in\mathbb{R},-\lambda\leq \zeta_{k,t,i,j}\leq \lambda $ such that
    \begin{equation}
    \label{lemmaeq}
        \begin{gathered}
             \eta \geq \sum_{t=1}^{\mathcal{T}} \sum_{i=1}^{n_\xi} \eta_{m,t,i}+h_m\\
     \eta_{m,t,i}\geq \boldsymbol{g}_{m,t,i}^T\boldsymbol{\varphi}^-_{t,i,j}-{\zeta}_{t,i,j} \Big(\boldsymbol{R}_{t,i}\boldsymbol{\varphi}^-_{t,i,j}-\widehat{{\xi}}_{t,i}\Big)\\
     \eta_{m,t,i}\geq \boldsymbol{g}_{m,t,i}^T\boldsymbol{\varphi}_{t,i,j-1}-{\zeta}_{t,i,j} \Big(\boldsymbol{R}_{t,i}\boldsymbol{\varphi}_{t,i,j-1}-\widehat{{\xi}}_{t,i}\Big)\\
     \forall m\in [M], \forall t\in [\mathcal{T}] , \forall i\in [n_\xi] ,\forall j\in [p_{t,i}],
        \end{gathered}
    \end{equation}
    where $\boldsymbol{\varphi}_{t,i,j}=G_{t,i}(w_{t,i,j})$, $\boldsymbol{\varphi}^-_{t,i,j}=G^-_{t,i}(w_{t,i,j})$, and $\boldsymbol{g}_m=(\boldsymbol{g}^T_{m,1,1},\cdots,\boldsymbol{g}^T_{m,1,n_\xi},\boldsymbol{g}^T_{m,2,1},\cdots,\boldsymbol{g}^T_{m,t,n_\xi})^T$.
\end{lem}


\begin{proof}
Constraint (\ref{maxi}) can be reformulated as
\begin{gather}
 \eta \geq \max_{\boldsymbol{\xi}\in \varXi}\bigg\{ f(\boldsymbol{\xi})-\lambda \left\| \boldsymbol{\xi}-\widehat{\boldsymbol{\xi}} \right\|\bigg\} \\
\iff \eta \geq 
\max_{\boldsymbol{\xi}\in \varXi}\bigg\{
\boldsymbol{g}_m^T G(\boldsymbol{\xi}) +h_m-\lambda \left\| \boldsymbol{\xi}-\widehat{\boldsymbol{\xi}} \right\|\bigg\},\forall m\in [M]\\
\iff \eta\geq 
\max_{\boldsymbol{\xi}^{\ast}\in \varXi^{\ast}}\bigg\{\boldsymbol{g}_m^T\boldsymbol{\xi}^\ast +h_m-\lambda \left\|\boldsymbol{R}\boldsymbol{\xi}^{\ast}-\widehat{\boldsymbol{\xi}} \right\| \bigg\},\forall m\in [M]\label{worst}\\\iff \eta\geq 
    \sum_{t=1}^{\mathcal{T}} \sum_{i=1}^{n_\xi}\max_{\boldsymbol{\xi}_{t,i}^{\ast}\in \varXi_{t,i}^{\ast}}\bigg\{ \boldsymbol{g}_{m,t,i}^T\boldsymbol{\xi}_{t,i}^{\ast}-\lambda \left|\boldsymbol{R}_{t,i}\boldsymbol{\xi}_{t,i}^{\ast}-\widehat{{\xi}}_{t,i} \right| \bigg\}\nonumber\\+h_m
    =\sum_{t=1}^{\mathcal{T}} \sum_{i=1}^{n_\xi} \eta_{m,t,i}+h_m,\forall m \in [M],\label{formu1}
\end{gather}
where $\eta_{m,t,i}=\max_{\boldsymbol{\xi}_{t,i}^{\ast}\in \varXi_{t,i}^{\ast}}\{ \boldsymbol{g}_{m,t,i}^T\boldsymbol{\xi}_{t,i}^{\ast}-\lambda |\boldsymbol{R}_{t,i}\boldsymbol{\xi}_{t,i}^{\ast}-\widehat{{\xi}}_{t,i} | \}$, and equality (\ref{worst}) holds due to the recovery relation (\ref{recover}), { and equality ({\ref{formu1}}) holds due to the property of the 1-norm and Assumption {\ref{supp_assum}} that the support $\varXi$ is a hyperrectangle.}

By leveraging the analytical formulation of $\varXi_{t,i}^{\ast}$ given in Proposition \ref{prop}, $\eta_{m,t,i}^s$ can be further reformulated as follows.
\begin{gather}
\eta_{m,t,i}=\max_{\boldsymbol{\xi}_{t,i}^{\ast}\in \varXi_{t,i}^{\ast}}\bigg\{ \boldsymbol{g}_{m,t,i}^T\boldsymbol{\xi}_{t,i}^{\ast}-\lambda \left|\boldsymbol{R}_{t,i}\boldsymbol{\xi}_{t,i}^{\ast}-\widehat{{\xi}}_{t,i} \right| \bigg\}\\
    =\max_{j\in [p_{t,i}]} \max_{\boldsymbol{\xi}_{t,i}^{\ast}\in K_{t,i,j}^\ast}\bigg\{ \boldsymbol{g}_{m,t,i}^T\boldsymbol{\xi}_{t,i}^{\ast}-\lambda \left|\boldsymbol{R}_{t,i}\boldsymbol{\xi}_{t,i}^{\ast}-\widehat{{\xi}}_{t,i} \right| \bigg\}\\
    \begin{gathered}
    =\max_{j\in [p_{t,i}]} \max_{\boldsymbol{\xi}_{t,i}^{\ast}\in K_{t,i,j}^\ast}\min_{|{\zeta}_{m,t,i,j}|\leq \lambda}\\ \bigg\{ \boldsymbol{g}_{m,t,i}^T\boldsymbol{\xi}_{t,i}^{\ast}-{\zeta}_{m,t,i,j} \bigg(\boldsymbol{R}_{t,i}\boldsymbol{\xi}_{t,i}^{\ast}-\widehat{{\xi}}_{t,i}\bigg)  \bigg\}
    \end{gathered}\label{eq3}\\
    \begin{gathered}
    =\max_{ j\in [p_{t,i}]}\min_{|\zeta_{m,t,i,j}|\leq \lambda} \max_{\boldsymbol{\xi}_{t,i}^{\ast}\in K_{t,i,j}^\ast} \\\bigg\{ \boldsymbol{g}_{m,t,i}^T\boldsymbol{\xi}_{t,i}^{\ast}-{\zeta}_{m,t,i,j} \bigg(\boldsymbol{R}_{t,i}\boldsymbol{\xi}_{t,i}^{\ast}-\widehat{{\xi}}_{t,i}\bigg)  \bigg\}\end{gathered} \label{eq1}\\
    \begin{gathered}
            =\max_{j\in [p_{t,i}]}\min_{|{\zeta}_{m,t,i,j}|\leq \lambda}\max_{\boldsymbol{\xi}_{t,i}^{\ast}\in \big\{G( w_{t,i,(j-1)} ),G^{-}(  w_{t,i,j} )\big\}}\\
    \Bigg\{ \boldsymbol{g}_{m,t,i}^T\boldsymbol{\xi}_{t,i}^{\ast}-{\zeta}_{m,t,i,j} \bigg(\boldsymbol{R}_{t,i}\boldsymbol{\xi}_{t,i}^{\ast}-\widehat{{\xi}}_{t,i}\bigg)\Bigg\},
    \end{gathered}\label{eq2}
\end{gather}
where equality (\ref{eq3}) comes from the definition of the dual norm, equality (\ref{eq1}) holds due to Sion’s minimax theorem \cite{sion1958general}, and equality (\ref{eq2}) comes from the continuity of the inner linear function. 

Finally, by combining (\ref{formu1}) and (\ref{eq2}), we can reach the equivalent reformulation (\ref{lemmaeq}).
\end{proof} 

{By combining Lemma \ref{lemma} with standard Wasserstein DRO techniques \cite{mohajerin2018data,zhou2021distributionally}}, we derive an equivalent reformulation of the objective function (\ref{obj2}) for the following three types of Wasserstein metric-based ambiguity sets.

\subsubsection{Wasserstein Ambiguity Set}
Wasserstein ambiguity set $\mathscr{B}$ defined in (\ref{wass}) is 
a closed ball centered at the empirical distribution  $ \widehat{\mathbb{P}}=\sum_{s=1}^N \delta_{\widehat{\boldsymbol{\xi}}^s}/N$,
where $ \delta_{(\cdot)}$ is  the Dirac delta distribution\cite{mohajerin2018data}.
\begin{equation}
    \label{wass}
    \begin{gathered}
    \mathscr{B}=\left\{ \mathbb{Q}\in \mathscr{P}(\varXi)\left| d_W\left(\mathbb{Q},\widehat{\mathbb{P}} \right)\leq \theta \right.  \right\},\\
    d_W\left(\mathbb{Q},\widehat{\mathbb{P}} \right)=\inf_{\kappa\in \lambda\left(\mathbb{Q},\widehat{\mathbb{P}}  \right)}
    \int_{\varXi}\left\| \boldsymbol{\xi}_{\mathbb{Q}},\boldsymbol{\xi}_{\widehat{\mathbb{P}}} \right\|  \ d\kappa \left( \boldsymbol{\xi}_{\mathbb{Q}}, \boldsymbol{\xi}_{\widehat{\mathbb{P}}}\right),
\end{gathered}
\end{equation}
where $\mathscr{P}(\varXi)$ is the set of all Borel probability measures supported on $\varXi$, $d_W(\cdot,\cdot)$ is the Wasserstein metric, and $\lambda(\mathbb{Q},\widehat{\mathbb{P}} )$ contains all Borel probability measures supported on $\varXi\times \varXi$ with marginal distributions $\mathbb{Q}$ and $\widehat{\mathbb{P}}$.

We develop a tractable equivalent reformulation for (\ref{obj2}) with respect to the Wasserstein ambiguity set in Theorem \ref{thmw}.
\begin{thm}
    \label{thmw}
    Let $\mathscr{B}$ be the Wasserstein ambiguity set defined in (\ref{wass}). Then, the objective function (\ref{obj2}) admits the following equivalent reformulation.
    \begin{gather}
    \begin{gathered}
\min_{\lambda,\eta^s,\eta^s_{k,t,i},\zeta_{k,t,i,j}^s,\boldsymbol{Y}_t,\boldsymbol{y}_t^0,\boldsymbol{Z}_t,\boldsymbol{z}_t^0}\hspace{12pt} \lambda \theta +\frac{1}{N}\sum_{s=1}^N \eta^s\hspace{45pt}\\
         \emph{s.t.} \hspace{25pt}   \eta^s \geq \sum_{t=1}^{\mathcal{T}} \sum_{i=1}^{n_\xi} \eta_{k,t,i}^s+r_k\\
     \eta_{k,t,i}^s\geq \boldsymbol{d}_{k,t,i}^T\boldsymbol{\varphi}^-_{t,i,j}-\zeta_{k,t,i,j}^{s} \Big(\boldsymbol{R}_{t,i}\boldsymbol{\varphi}^-_{t,i,j}-\widehat{{\xi}}^s_{t,i}\Big)\\
     \eta_{k,t,i}^s\geq \boldsymbol{d}_{k,t,i}^T\boldsymbol{\varphi}_{t,i,j-1}-\zeta_{k,t,i,j}^{s} \Big(\boldsymbol{R}_{t,i}\boldsymbol{\varphi}_{t,i,j-1}-\widehat{{\xi}}^s_{t,i}\Big)\\
     |  \zeta_{k,t,i,j}^s | \leq\lambda  \\
     \forall s\in [N], \forall k\in [K], \forall t\in [\mathcal{T}] , \forall i\in [n_\xi] ,\forall j\in [p_{t,i}],\label{refwass}
     \end{gathered}
    \end{gather}
    where $\boldsymbol{\varphi}_{t,i,j}=G_{t,i}(w_{t,i,j})$, $\boldsymbol{\varphi}^-_{t,i,j}=G^-_{t,i}(w_{t,i,j})$, and $\boldsymbol{d}_k=(\boldsymbol{d}^T_{k,1,1},\cdots,\boldsymbol{d}^T_{k,1,n_\xi},\boldsymbol{d}^T_{k,2,1},\cdots,\boldsymbol{d}^T_{k,t,n_\xi})^T$.
\end{thm}
\begin{proof}
     The worst-case expectation term in (\ref{obj2}) can be equivalently reformulated as follows \cite{mohajerin2018data}.
        \begin{gather}
        \max_{\mathbb{P}\in \mathscr{B}}\mathbb{E}_{\boldsymbol{\xi}\sim\mathbb{P}}\bigg[ \max_{k\in [K]}\bigg\{ \boldsymbol{d}_k^TG(\boldsymbol{\xi})+r_k \bigg\} \bigg]\\
          =\min_{\lambda \geq 0}\hspace{5pt} \lambda \theta + \frac{1}{N}\sum_{s=1}^N \eta^s,
              \text{ s.t. }\eta^s\geq \max_{\boldsymbol{\xi}\in \varXi} \bigg\{ \max_{k\in [K]}\big\{ \boldsymbol{d}_k^TG(\boldsymbol{\xi})+r_k\big\}\nonumber\\-\lambda \left\| \boldsymbol{\xi}-\widehat{\boldsymbol{\xi}}^s \right\|  \bigg\},\forall s\in [N].
     \label{gao1}
\end{gather}

Subsequently, by applying Lemma \ref{lemma} to the constraint of the (\ref{gao1}), the reformulation (\ref{refwass}) is derived.
\end{proof}

{
\begin{rem}\label{rem11}
    As mentioned in Section {\ref{intro}},\mbox{\cite{feng2021multistagea}} also investigated the Wasserstein DRO problem with the lifted decision rule. However, their reformulation method is actually based on a relaxation of the ambiguity set, while our method is an equivalent reformulation and has the same computational complexity as\mbox{\cite{feng2021multistagea}}. To see this point more clearly, in Appendix \ref{appendix_super}, we prove that our method is superior to the method in\mbox{\cite{feng2021multistagea}}.
\end{rem}}

\subsubsection{Mixed Wasserstein-Moment Ambiguity Set}

To enhance the conventional Wasserstein ambiguity set, constructing a mixed ambiguity set has been favored in recent years, by leveraging other information to remove unrealistic probability distributions in the Wasserstein ball \cite{zhou2021distributionally}. Therefore, we consider a mixed Wasserstein-moment ambiguity set as follows.
\begin{equation}
\label{wassm}
    \mathscr{B}=\left\{ \mathbb{Q}\in \mathscr{P}(\varXi)\left| \  \begin{gathered}
        d_W\left(\mathbb{Q},\widehat{\mathbb{P}} \right)\leq \theta \\
        \underline{\boldsymbol{\xi}} \leq \mathbb{E}_{\mathbb{Q}}[\boldsymbol{\xi}]\leq \overline{\boldsymbol{\xi}}
    \end{gathered} \right.  \right\},
\end{equation}
where $\underline{\boldsymbol{\xi}}$ and $\overline{\boldsymbol{\xi}}$ are the lower bound and upper bound of first-order moment, respectively.

The reformulation with respect to this mixed Wasserstein-moment ambiguity set is presented in Theorem \ref{thmwm}.
\begin{thm}
    \label{thmwm}
    Let $\mathscr{B}$ be the mixed Wasserstein-moment ambiguity set defined in (\ref{wassm}). Then, the objective function (39) admits the following equivalent reformulation.
    \begin{gather}
     \label{refwassm}
     \begin{gathered}
\min_{\lambda,\eta^s,\eta^s_{k,t,i},\underline{\boldsymbol{\beta}},\overline{\boldsymbol{\beta}},\zeta_{k,t,i,j}^s,\boldsymbol{Y}_t,\boldsymbol{y}_t^0,\boldsymbol{Z}_t,\boldsymbol{z}_t^0}\hspace{10pt}  \lambda \theta +\frac{1}{N}\sum_{s=1}^N \eta^s-\underline{\boldsymbol{\beta}}^T\underline{\boldsymbol{\xi}}+\overline{\boldsymbol{\beta}}^T\overline{\boldsymbol{\xi}}\hspace{28pt}\\
         \emph{s.t.} \hspace{25pt}   \eta^s \geq \sum_{t=1}^{\mathcal{T}} \sum_{i=1}^{n_\xi} \eta_{k,t,i}^s+r_k\hspace{35pt}\\
     \eta_{k,t,i}^s\geq \tilde{\boldsymbol{d}}_{k,t,i}^T\boldsymbol{\varphi}^-_{t,i,j}-\zeta_{k,t,i,j}^{s} \Big(\boldsymbol{R}_{t,i}\boldsymbol{\varphi}^-_{t,i,j}-\widehat{{\xi}}^s_{t,i}\Big)\\
     \eta_{k,t,i}^s\geq \tilde{\boldsymbol{d}}_{k,t,i}^T\boldsymbol{\varphi}_{t,i,j-1}-\zeta_{k,t,i,j}^{s} \Big(\boldsymbol{R}_{t,i}\boldsymbol{\varphi}_{t,i,j-1}-\widehat{{\xi}}^s_{t,i}\Big)\\
     |  \zeta_{k,t,i,j}^s | \leq\lambda,\boldsymbol{0}\leq\underline{\boldsymbol{\beta}}, \boldsymbol{0}\leq \overline{\boldsymbol{\beta}}\\
         \forall s\in [N], \forall k\in [K], \forall t\in [\mathcal{T}] , \forall i\in [n_\xi] ,\forall j\in [p_{t,i}],
          \end{gathered}
    \end{gather}
    where $\tilde{\boldsymbol{d}}_k=\boldsymbol{d}_k+\boldsymbol{R}^T\underline{\boldsymbol{\beta}}-\boldsymbol{R}^T\overline{\boldsymbol{\beta}}$, $\boldsymbol{\varphi}_{t,i,j}=G_{t,i}(w_{t,i,j})$, and $\boldsymbol{\varphi}^-_{t,i,j}=G^-_{t,i}(w_{t,i,j})$.
\end{thm}
\begin{proof}
    { The reformulation of the worst-case expectation with respect to the mixed Wasserstein-moment ambiguity set has been developed by {\cite{zhou2021distributionally}} as follows.}
    \begin{gather}
        \max_{\mathbb{P}\in \mathscr{B}}\mathbb{E}_{\boldsymbol{\xi}\sim\mathbb{P}}\bigg[ \max_{k\in [K]}\bigg\{ \boldsymbol{d}_k^TG(\boldsymbol{\xi})+r_k \bigg\} \bigg]\\=\left\{ \begin{gathered} \min_{0\leq 
\lambda,\boldsymbol{0}\leq\underline{\boldsymbol{\beta}}, \boldsymbol{0}\leq \overline{\boldsymbol{\beta}}} \lambda \theta +\frac{1}{N}\sum_{s=1}^N \eta^s-\underline{\boldsymbol{\beta}}^T\underline{\boldsymbol{\xi}}+\overline{\boldsymbol{\beta}}^T\overline{\boldsymbol{\xi}}\hspace{28pt}\\
\text{s.t. }\hspace{0pt}
     \begin{gathered}\eta^s\geq \max_{\boldsymbol{\xi}\in \varXi} \bigg\{ \max_{k\in [K] }\big\{ \big(\boldsymbol{d}_k+\boldsymbol{R}^T\underline{\boldsymbol{\beta}}-\boldsymbol{R}^T\overline{\boldsymbol{\beta}}\big)^TG(\boldsymbol{\xi})\\+r_k\big\}-\lambda \left\| \boldsymbol{\xi}-\widehat{\boldsymbol{\xi}}^s \right\|  \bigg\},\forall s\in [N].\end{gathered} \end{gathered} \right.\label{wm1}
    \end{gather}

    By applying Lemma 1 to the constraint of (\ref{wm1}), we can reach to the reformulation (\ref{refwassm}).
\end{proof}

\subsubsection{Event-Wise Wasserstein Ambiguity Set}
In some cases, the disturbance distribution is affected by certainty events, such as the distribution of daily average temperature depending on the season. Therefore, constructing an event-wise ambiguity set can characterize such disturbance more precisely \cite{chen2020robust}.

We consider an event with $L$ possible realizations. Conditioning on the realization of a scenario $l\in[L]$, the disturbance has support $\varXi^l$ and $N_l$ historical data $\widehat{\boldsymbol{\xi}}^{l,s},s\in[N_l]$ is available. Based on this information, we construct the following event-wise Wasserstein ambiguity set.
\begin{equation}
\label{event}
    \mathscr{B}=\left\{ \mathbb{Q}\in \mathscr{P}\Big(\mathbb{R}^{\sum_{t=1}^{\mathcal{T}} n_\xi}\times [L]\Big)\left|  \begin{gathered}
        (\tilde{\boldsymbol{\xi}},\tilde{l})\sim \mathbb{Q}\\
        \mathbb{Q}(\tilde{l}=l)=p_l,\forall l\in [L]\\
        \mathbb{Q}\Big(\tilde{\boldsymbol{\xi}}\in \varXi^l|\tilde{l}=l\Big)=1\\ \hspace{70pt} ,\forall l\in [L]\\
        d_W\Big(\mathbb{Q}(\tilde{\boldsymbol{\xi}}|\tilde{l}=l),\widehat{\mathbb{P}}^l\Big)\leq \theta_l\\\hspace{70pt} ,\forall l\in [L]
    \end{gathered} \right.  \right\},
\end{equation}
where $\widehat{\mathbb{P}}^l=\sum_{s=1}^{N_l}
\delta_{\widehat{\boldsymbol{\xi}}^{l,s}}/N_l$ and 
$p_l$ is the probability of scenario $l$ with $\sum_{l=1}^Lp_l=1$.

To adapt to this event-wise disturbance distribution, we construct distinct LCPs for different scenarios. For scenario $l$, we leverage the lifting function $G^l$ with breakpoints $w_{t,i,j}^l,j\in [p_{t,i}^l-1]$ to construct the LCP  for continuous and integer controls with optimization variables $\boldsymbol{Y}_t^l,\boldsymbol{y}_t^{l,0},\boldsymbol{Z}_t^l,\boldsymbol{z}_t^{l,0}$.
Based on the event-wise LCP, the objective (39) is equivalent to
\begin{gather}
\min_{\boldsymbol{Y}_{t}^l,\boldsymbol{y}_t^{l,0},\boldsymbol{Z}_t^l,\boldsymbol{z}_t^{l,0}}\hspace{0pt}\max_{\mathbb{P}\in \mathscr{B}}\mathbb{E}_{(\tilde{\boldsymbol{\xi}},\tilde{l})\sim\mathbb{P}}\Bigg[ \sum_{l=1}^L \mathbbm{1}_{\tilde{l}=l}\max_{k\in [K]}\bigg\{ \Big(\boldsymbol{d}_k^{l}\Big)^ TG^l(\boldsymbol{\xi})+r_k^l \bigg\} \Bigg]\label{event1}\\
    \text{s.t.}\hspace{38pt}\boldsymbol{m}_l\geq \boldsymbol{E}_l^TG^l(\boldsymbol{\xi}),\forall \boldsymbol{\xi}\in \varXi^l\hspace{42pt}\label{event2}\\
    \boldsymbol{Y}_{t}^l\in \mathbb{R}^{n_{u}\times n_{V_t^l}},y_t^{l,0}\in \mathbb{R}^{n_{u}},\boldsymbol{Z}_t^l\in \mathbb{Z}^{n_{\gamma}\times n_{Q_t^l}},z_t^{l,0}\in \mathbb{Z}^{n_{\gamma}},
\end{gather}
where $\mathbbm{1}$ is the indicator function and $\boldsymbol{d}_k^{l},r_k^l,\boldsymbol{m}_l,\boldsymbol{E}_l$ can be computed according to (37)-(38).

The equivalent reformulation of the objective function (\ref{event1}) is developed in Theorem \ref{thm3}.
\begin{thm}
    \label{thm3}
        Let $\mathscr{B}$ be the event-wise Wasserstein ambiguity set defined in (\ref{event}). Then, the objective function (\ref{event1}) admits the following equivalent reformulation.
        \begin{gather}
\begin{gathered}
\min_{\lambda_l,\eta^{l,s},\eta^{l,s}_{k,t,i},\zeta_{k,t,i,j}^{l,s},\boldsymbol{Y}_{t}^l,\boldsymbol{y}_t^{l,0},\boldsymbol{Z}_t^l,\boldsymbol{z}_t^{l,0}}\hspace{10pt} \sum_{l=1}^Lp_l\Bigg[\lambda_l \theta_l +\frac{1}{N_l}\sum_{s=1}^{N_l} \eta^{l,s}\Bigg]\hspace{25pt}\\
        \emph{s.t.} \hspace{30pt}   \eta^{l,s} \geq \sum_{t=1}^{\mathcal{T}} \sum_{i=1}^{n_\xi} \eta_{k,t,i}^{l,s}+r_k^l\hspace{15pt}\\
     \eta_{k,t,i}^{l,s}\geq \boldsymbol{d}_{k,t,i}^{l\ T}\boldsymbol{\varphi}^{l-}_{t,i,j}-\zeta_{t,i,j}^{l,s} \Big(\boldsymbol{R}_{t,i}\boldsymbol{\varphi}^{l-}_{t,i,j}-\widehat{{\xi}}^{l,s}_{t,i}\Big)\\
     \eta_{k,t,i}^{l,s}\geq \boldsymbol{d}_{k,t,i}^{l\ T}\boldsymbol{\varphi}^l_{t,i,j-1}-\zeta_{t,i,j}^{l,s} \Big(\boldsymbol{R}_{t,i}\boldsymbol{\varphi}^l_{t,i,j-1}-\widehat{{\xi}}^{l,s}_{t,i}\Big)\\
     |  \zeta^{l,s}_{k,t,i,j} | \leq\lambda \\
    \forall l\in [L], \forall s\in [N_l], \forall k\in [K], \\\forall t\in [\mathcal{T}] , \forall i\in [n_\xi] ,\forall j\in [p_{t,i}^l],  \label{refevent}
    \end{gathered} 
        \end{gather}
        where $\boldsymbol{\varphi}^l_{t,i,j}=G^l_{t,i}(w_{t,i,j})$ and $\boldsymbol{\varphi}^{l-}_{t,i,j}=G^{l-}_{t,i}(w_{t,i,j})$.
\end{thm}
\begin{proof}
    Notice that the worst-case expectation is equivalent to
    \begin{gather}
       \sum_{l=1}^L p_l  \max_{\mathbb{P}\in \mathscr{B}_l}\mathbb{E}_{\boldsymbol{\xi}\sim\mathbb{P}}\bigg[ \max_{k\in [K]}\bigg\{ \Big(\boldsymbol{d}_k^{l}\Big)^TG^l(\boldsymbol{\xi})+r_k^l \bigg\} \bigg],
    \end{gather}
    where 
\begin{equation}
        \mathscr{B}_l=\left\{ \mathbb{Q}\in \mathscr{P}(\varXi^l)\left| d_W\left(\mathbb{Q},\widehat{\mathbb{P}}^l \right)\leq \theta_l \right.  \right\}.
\end{equation}

Therefore, by applying Theorem 3 to the inner worst-case expectation term with respect to $\mathscr{B}_l$, the reformulation (\ref{refevent}) is derived.
\end{proof}

\subsection{Equivalent Reformulation of Constraint}
\label{equi_cons_sec}
{Notice that constraint (\ref{cons2}) is also in the form of Lemma \ref{lemma} with $\lambda=0$. Therefore, we have the following corollary to reformulate (\ref{cons2}).
\begin{cor}\label{coro_const}
    Constraint (\ref{cons2}) can be equivalently reformulated as: $\exists \boldsymbol{m}_{t,i}$ (decision vectors of compatible dimension) such that
    \begin{equation}
        \begin{gathered}
        \boldsymbol{m}\geq \sum_{t=1}^{\mathcal{T}} \sum_{i=1}^{n_\xi} \boldsymbol{m}_{t,i},\emph{and }\forall t\in [\mathcal{T}] , \forall i\in [n_\xi] ,\forall j\in [p_{t,i}],\\
        \boldsymbol{m}_{t,i}\geq \boldsymbol{E}_{t,i}G^-_{t,i}(w_{t,i,j}),
     \boldsymbol{m}_{t,i}\geq \boldsymbol{E}_{t,i}G_{t,i}(w_{t,i,j-1}),
    \end{gathered}
    \end{equation}
    where $\boldsymbol{E}=(\boldsymbol{E}_{1,1},\cdots,\boldsymbol{E}_{1,n_\xi},\boldsymbol{E}_{2,1},\cdots,\boldsymbol{E}_{\mathcal{T},n_\xi})$.
\end{cor}}


\section{Numerical Experiment}
\label{case}

To test the performance of the proposed DR-LCP method, we conduct experiments on a classic single-item inventory control problem under the Wasserstein ambiguity set and compare proposed DR-LCP method with the ACP method \cite{van2015distributionally,goulart2006optimization} and the relaxed Wasserstein LCP (RW-LCP) method \cite{feng2021multistagea} in the literature.

The inventory control problem under the uncertainty of demand is presented in (\ref{inven}).
\begin{equation}
\label{inven}
\begin{gathered}
        \min_{{u}_{t}(\cdot),{\gamma}_{t,k}(\cdot)} \max_{\mathbb{P}\in \mathscr{B}}\
    \mathbb{E}_{\mathbb{P}} \left\{ 
    \begin{gathered}
            \sum_{t=1}^{\mathcal{T}}\bigg[ax_t(\boldsymbol{\xi}_{[t]})+ bu_{t-1}(\boldsymbol{\xi}_{[t-1]})\\+\sum_{k=1}^Kc_{k}q_k\gamma_{t,k}(\boldsymbol{\xi}_{[t]}) \bigg] 
    \end{gathered}\right\}\\
    \text{s.t.}\hspace{15pt} x_t(\boldsymbol{\xi}_{[t]})=x_{t-1}(\boldsymbol{\xi}_{[t-1]})+u_{t-1}(\boldsymbol{\xi}_{[t-1]})\\ \hspace{60pt}+\sum_{k=1}^Kq_k\gamma_{t,k}(\boldsymbol{\xi}_{[t]})-\boldsymbol{\xi}_t,\forall t\in [\mathcal{T}]\\
    \hspace{20pt} x_t(\boldsymbol{\xi}_{[t]})\geq 0,
    u_{t-1}(\boldsymbol{\xi}_{[t-1]})\geq 0,\gamma_{t,k}(\boldsymbol{\xi}_{[t]})\in \{0,1\},
    \\\hspace{30pt}\forall t\in [\mathcal{T}],\forall k\in [K],
\end{gathered}
\end{equation}
where $\mathscr{B}$ is the Wasserstein ambiguity set. $x_t$ is the state of the system and represents the inventory level at the end of time $t$, with the initial inventory level $x_0$. $u_{t-1}(\boldsymbol{\xi}_{[t-1]})$ and ${\gamma}_{t,k}(\boldsymbol{\xi}_{[t]})$ are continuous and binary controls, respectively. At time $t$, continuous ordering decision $u_{t-1}\geq 0 $ is made in advance before the realization of demand $\boldsymbol{\xi}_t$, incurring booking cost $bu_{t-1}$. After the demand $\boldsymbol{\xi}_t$ is realized, the binary decisions ${\gamma}_{t,k}$ for lot ordering with fixed order amount $q_k$ at unit price $c_k$ ($c_k>b$) are made to make up for the unsatisfied demand. Finally, the remaining inventory $x_t$ will be held at unit holding price $a$.

In constructing the LCP for $u_{t-1}$ and ${\gamma}_{t,k}$, we set the segment number $p_{t,i}$ to the same value $p$ for all time step $t$ and set the breakpoints to be $p$ equal division points of the support. {In the Wasserstein ambiguity set, $20$ samples are used to construct the empirical distribution, and the Wasserstein radius is set to the Wasserstein distance between the empirical distribution and the sample-generating distribution.}

All experiments are conducted on a laptop with Intel i5 CPU and 32G memory. The control problems are established by YALMIP in MATLAB R2022a and solved by GUROBI 9.5.2, {where the solution optimality gap is set to 0.1\%.}

We present the optimal objective value and computational time of different methods under different control horizon $T$ and segment number $p$ in Table \ref{objv}, where the computational time limit is set to 1000 seconds for each case. 
{For parameters in this experiment, we set the uncertainty support to $[20,100]^{\mathcal{T}}$, and take $K=3$, $a=5$, $b=2$, $c_k=5$, and $q_k=30,k\in [K]$. Uncertain demand at time $t$ is sampled independently from a Gaussian distribution $\mathcal{N}(\text{Mean}_t,1)$, where the mean of all 16 time stages are $\text{Mean}=(50,30,70,30,50,30,70,30,50,30,\\70,30,50,30,70,30)$.} From Table \ref{objv}, we observe 
that both LCP methods can derive an approximate solution in a reasonable computational time and the optimal objective value decreases as the number of segments grows. When segment number $p>1$, the proposed DR-LCP method always provides a lower objective value than the RW-LCP method, and this superiority in objective value grows as the control horizon $\mathcal{T}$ expands. More precisely, our method demonstrates a maximum improvement of 30.0\% in objective value compared to the RW-LCP method, indicating the effectiveness of the proposed DR-LCP, especially in a long control horizon.

\begin{table}[hbtp]
  \vspace{-0pt}
  \begin{center}
  \begin{threeparttable}
  \caption{Objective Values (\$) and Computational Times (s) of Different Methods}
  \label{objv}
\begin{tabular}{cccccc}
  \toprule
  \multirow{2}{*}[-1pt]{$\mathcal{T}$} &\multirow{2}{*}[-1pt]{Seg. No. $p$} &\multicolumn{2}{c}{Objective Value (\$)}&\multicolumn{2}{c}{Comput. Time (s)} \\
  \cmidrule(r){3-4}\cmidrule(r){5-6}\\
  \addlinespace[-10pt]
  &&RW-LCP&DR-LCP&RW-LCP&DR-LCP\\
  \midrule
\multirow{3}{*}[-5pt]{2} & 1  &	 903.6&   903.6& 0.04  &  0.04 \\
&2	&787.6&521.5&0.05&0.05 \\
&4	&536.9 &331.5 &0.11& 0.08\\
&8	&360.7&331.5&  0.35&0.30\\
\midrule
\multirow{3}{*}[-5pt]{4} & 1  &1609.5 & 1609.5 &0.04& 0.04 \\
&2	&1462.6 & 1142.8&0.07&0.06  \\
&4	&1079.2  & 755.7&0.32&0.14 \\
&8	&825.9& 755.7&2.37&0.99\\
\midrule
\multirow{3}{*}[-5pt]{8} & 1  &3081.2 & 3081.2 &  0.05&0.05\\
&2	&2898.0& 2327.5&0.16 &0.19  \\
&4	&2208.3 & 1758.8&1.71&1.07\\
&8	& 1775.4&1653.5&47.16&20.52\\
\midrule
\multirow{4}{*}[-5pt]{16} & 1  &6023.2 &6023.2 &0.15&0.13\\
&2	&5768.1& 4864.0  &1.32&0.95 \\
&4	&4808.0 & 4009.3 &22.86&37.59\\
&\multirow{2}{*}[0pt]{8} 	&3787.1 &3546.5  &\multirow{2}{*}[0pt]{1000.00}& \multirow{2}{*}[0pt]{1000.00}\\
& & (6.3\% gap) &   (7.2\% gap) &&\\
  \bottomrule
\end{tabular}
\begin{tablenotes}
      \small
      \item[\textdagger] The optimal upper bound of objective value is presented for cases that reach the computational time limit.
    \end{tablenotes}
\end{threeparttable}
\end{center}
\vspace{-10pt}
\end{table}

{To test the performance of the proposed DR-LCP method in problems with more integer controls, we vary the number of integer controls $K$ from $4$ to $32$. For parameters, we set $\mathcal{T}=4$, $p=4$, and $q_k=80/K, k\in [K]$, and other parameters are the same as in the first experiment. The results are presented in Table \ref{bin_table}, which shows that as more integer controls are available, the objective values of both the RW-LCP method and the proposed DR-LCP method decrease. 
It is also evident that in all cases, our DR-LCP method consistently outperforms the RW-LCP method in objective values while maintaining moderate computational time.}

\begin{table}[hbtp]
  \vspace{-0pt}
  \begin{center}
  \begin{threeparttable}
  \caption{Objective Values (\$) and Computational Times (s) of Different Number of Integer Controls}
  \label{bin_table}
  \begin{tabular}{ccccc}
  \toprule
  \multirow{2}{*}[-1pt]{$K$}  &\multicolumn{2}{c}{Objective Value (\$)}&\multicolumn{2}{c}{Comput. Time (s)} \\
  \cmidrule(r){2-3}\cmidrule(r){4-5}\\
  \addlinespace[-10pt]
  &RW-LCP&DR-LCP&RW-LCP&DR-LCP\\
  \midrule
4	&1011.0&694.4&0.20&0.11 \\
8	&999.6 &694.4 &0.23& 0.18\\
16	&989.1&689.3&  0.31&0.26\\
32	&989.1&  689.3&  0.53&0.48\\
  \bottomrule
\end{tabular}
\end{threeparttable}
\end{center}
\vspace{-10pt}
\end{table}


{We further implement the proposed method in the shrinking horizon fashion and test its closed-loop performance. Specifically, after the realization of the uncertain demand $\xi_t$, (\ref{inven}) is solved, where the immediate controls $u_t(\boldsymbol{\xi}_{[t]}),\gamma_{t,k}(\boldsymbol{\xi}_{[t]}),k\in [K]$ are optimized over $\mathbb{R}$ and $\{0,1\}$ (since $\boldsymbol{\xi}_{[t]}$ is realized) and controls at later time steps (\emph{i.e.}, $u_{t+1}(\boldsymbol{\xi}_{[t+1]}),\cdots,u_{T-1}(\boldsymbol{\xi}_{[T-1]})$ and $\gamma_{t+1,k}(\boldsymbol{\xi}_{[t+1]}),\cdots,\gamma_{T,k}(\boldsymbol{\xi}_{[T]})$) are optimized over the LCP. After solving (\ref{inven}), the immediate controls $u_t$ and $\gamma_{t,k}(\boldsymbol{\xi}_{[t]})$ are implemented, and then uncertainty $\xi_{t+1}$ is realized. The closed-loop experiment is conducted on the problem with control horizon $\mathcal{T}=4$ and all other parameters are the same as in the first experiment, and the results of 100 simulations are presented in Fig. \ref{Fig_close}. In Fig. \ref{Fig_close}, the proposed DR-LCP method outperforms the RW-LCP method by 15.7\% in the case with the segment number $p=8$, which demonstrates the superiority of the proposed method in closed-loop performance.}
\begin{figure}[!htp]  
\centering 
\includegraphics[width=0.48\textwidth]{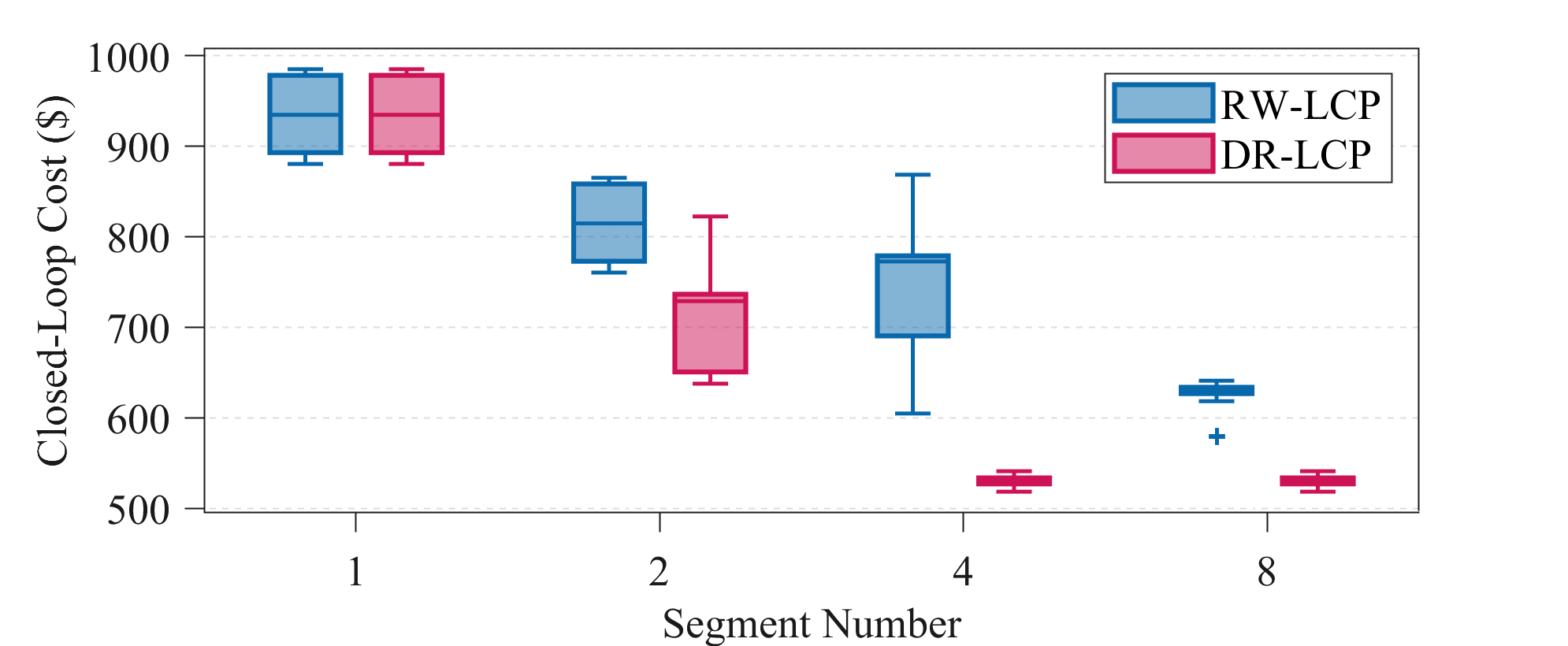} \vspace{-8pt}
{\caption{Closed-loop performance.}\label{Fig_close}
}
\vspace{-15pt}
\end{figure}

\section{Conclusion} 
\label{concl}
This work investigated the {finite-horizon} DRMIC of uncertain linear systems. We proposed a DR-LCP method to obtain an approximate solution to
the DRMIC problem for Wasserstein metric-based ambiguity sets by optimizing over the family of the LCPs. { We then conducted an asymptotic performance analysis and established a tight non-asymptotic performance bound of the proposed DR-LCP method.}
 Subsequently, we developed a solution methodology to find the optimal LCP for the DRMIC problem. Finally, numerical experiments on an inventory control problem demonstrated the effectiveness of the proposed DR-LCP method compared with existing methods in the literature.

{To alleviate the computational burden of long-horizon problems, distributed algorithms \cite{cherukuri2019cooperative} may be developed for our method to accelerate solving the reformulated optimization problem as future work.}
\vspace{-8pt}
{

\section*{Acknowledgement}

The authors would like to thank the editor and all the reviewers for their valuable and constructive comments, which help improve this paper.}

\vspace{-10pt}
 \appendix

\subsection{Superiority to\mbox{\cite{feng2021multistagea}}}
\label{appendix_super}
In \cite{feng2021multistagea}, the Wasserstein ambiguity set $\mathscr{B}$ (\ref{wass}) is relaxed to $\mathscr{B}_\text{Lifted}$ as follows.
\begin{gather}\allowdisplaybreaks\label{lifted_amb}
        \mathscr{B}_\text{Lifted}=\left\{ \mathbb{Q}\in \mathscr{P}(\text{conv}({\varXi^\ast}))\Big| d_W\Big(\mathbb{Q},\widehat{\mathbb{P}}_\text{Lifted}\Big)\leq \theta^{\ast} \right\},\\
    \theta^{\ast}=\sup_{\mathbb{P}\in \mathscr{B}} d_W(\mathbb{P}\circ G^{-1},\widehat{\mathbb{P}}_\text{Lifted}),\label{lifted_amb1}
\end{gather}
where $\widehat{\mathbb{P}}_\text{Lifted}=\sum_{s=1}^N\delta_{G(\widehat{\boldsymbol{\xi}}^s)}/N$ and $\mathbb{P}\circ G^{-1}$ is the push-forward probability measure by lifting function $G$.

By (\ref{lifted_amb}), the lifted ambiguity set $\mathscr{B}_\text{Lifted}$ `contains' the standard Wasserstein ambiguity set $\mathscr{B}$. Hence it can be viewed as a conservative approximation of $\mathscr{B}$.

Based on $\mathscr{B}_\text{Lifted}$, \cite{feng2021multistagea} resorts to optimizing the following objective.
\begin{equation}\label{relax_obj}
    \min_{\boldsymbol{Y}_t,\boldsymbol{y}_t^0,\boldsymbol{Z}_t,\boldsymbol{z}_t^0}\hspace{8pt}\max_{\mathbb{P}^\ast\in \mathscr{B}_\text{Lifted}}\mathbb{E}_{\boldsymbol{\xi}^\ast\sim\mathbb{P}^\ast}\bigg[ \max_{k\in [K]}\bigg\{ \boldsymbol{d}_k^T\boldsymbol{\xi}^\ast+r_k \bigg\} \bigg].
\end{equation}

We prove in Corollary \ref{coro1} that (\ref{relax_obj}) is a relaxation of (\ref{obj2}), which implies that by directly optimizing over (\ref{obj2}) our method is always better than the method in \cite{feng2021multistagea}.
\begin{cor}\label{coro1}
    For any $\boldsymbol{Y}_t,\boldsymbol{y}_t^0,\boldsymbol{Z}_t,$ and $\boldsymbol{z}_t^0$, it holds that 
    \begin{gather}
    \max_{\mathbb{P}^\ast\in \mathscr{B}_\text{Lifted}}\mathbb{E}_{\boldsymbol{\xi}^\ast\sim\mathbb{P}^\ast}\bigg[ \max_{k\in [K]}\bigg\{ \boldsymbol{d}_k^T\boldsymbol{\xi}^\ast+r_k \bigg\} \bigg]\label{worst_rela}\\
    \geq
        \max_{\mathbb{P}\in \mathscr{B}}\mathbb{E}_{\boldsymbol{\xi}\sim\mathbb{P}}\bigg[ \max_{k\in [K]}\bigg\{ \boldsymbol{d}_k^TG(\boldsymbol{\xi})+r_k \bigg\} \bigg].
    \end{gather}
\end{cor}
\begin{proof}
     $\forall \mathbb{P}\in \mathscr{B}$, we have
    \begin{gather}
        \mathbb{E}_{\boldsymbol{\xi}\sim\mathbb{P}}\bigg[ \max_{k\in [K]}\bigg\{ \boldsymbol{d}_k^TG(\boldsymbol{\xi})+r_k \bigg\} \bigg]=\mathbb{E}_{\boldsymbol{\xi}^\ast\sim\mathbb{P}_G}\bigg[ \max_{k\in [K]}\bigg\{ \boldsymbol{d}_k^T\boldsymbol{\xi}^\ast+r_k \bigg\} \bigg]\nonumber\\\leq  \max_{\mathbb{P}^\ast\in \mathscr{B}_\text{Lifted}}\mathbb{E}_{\boldsymbol{\xi}^\ast\sim\mathbb{P}^\ast}\bigg[ \max_{k\in [K]}\bigg\{ \boldsymbol{d}_k^T\boldsymbol{\xi}^\ast+r_k \bigg\} \bigg].\label{side_both}
    \end{gather}

    Therefore, Corollary \ref{coro1} holds by taking maximum on both sides of (\ref{side_both}).
\end{proof}

\FloatBarrier
\bibliographystyle{IEEEtran}
\bibliography{ref.bib}

\end{document}